\documentclass[11pt]{amsart}

\usepackage{color}
\def\rojo{}
\def\red{}
\def\blue{}

\usepackage{enumerate}
\usepackage{tikz}
\usetikzlibrary{arrows.meta,graphs}
\usepackage{oubraces}
\usepackage[all]{xy}
\usepackage{amsmath,amssymb}
\usepackage{graphicx}
\usepackage[mathscr]{euscript}
\usepackage{chngcntr}
\counterwithin{figure}{section}

\newtheorem{theorem}{Theorem}[section]
\newtheorem{proposition}[theorem]{Proposition}
\newtheorem{lemma}[theorem]{Lemma}
\newtheorem{corollary}[theorem]{Corollary}

\newtheorem{definition}[theorem]{Definition}

\theoremstyle{remark}
\newtheorem{example}[theorem]{Example}
\newtheorem{remark}[theorem]{Remark}

\newtheorem{question}[theorem]{Question}

\newcommand{\figref}[1]{Figure~\ref{#1}}

\newcommand{\thmref}[1]{Theorem~\ref{#1}}
\newcommand{\propref}[1]{Proposition~\ref{#1}}
\newcommand{\lemref}[1]{Lemma~\ref{#1}}
\newcommand{\corref}[1]{Corollary~\ref{#1}}

\newcommand{\remref}[1]{Remark~\ref{#1}}

\newcommand{\norma}[1]{\mathcal{N}^{#1}(K)}

\newcommand{\real}[1]{\mathbb{R}^{#1}}
\newcommand{\texto}[1]{\hspace{1mm}\text{#1}\hspace{1mm}}

\def\cat{{\sf{cat}}}
\def\zcl{{\sf{zcl}}}
\def\cupl{{\sf{cup}}}

\def\dim{\mathrm{dim}}

\def\Z{\mathbb Z}

\def\RP{P}

\def\TC{{\sf{TC}}}

\def\secat{{\sf{secat}}}
\def\F{\mathcal{F}}

\def\imm{\protect\operatorname{Imm}}
\def\CW{\protect\operatorname{CW}}

\newcommand{\und}{\underline}

\newcommand{\cd}{{\rm {cd}}}

\newcommand\tc{{\sf{TC}}}

\allowdisplaybreaks
	
\begin{document}

\title[\red{RAA groups, Polyhedral Products} and $\TC$-Generating Function]{Right-Angled Artin Groups, \red{Polyhedral Products} and the $\TC$-Generating Function}

\author{Jorge Aguilar-Guzm\'an}
\address{Departamento de Matem\'aticas\\
Centro de Investigaci\'on y de Estudios Avanzados del I.P.N.\\
M\'exico City 07000\\M\'exico}
\email{jaguzman@math.cinvestav.mx}

\author{Jes\'us Gonz\'alez}
\address{Departamento de Matem\'aticas\\
Centro de Investigaci\'on y de Estudios Avanzados del I.P.N.\\
M\'exico City 07000\\M\'exico}
\email{jesus@math.cinvestav.mx}

\author{John Oprea}
\address{Department of Mathematics \\
Cleveland State University\\
Cleveland, Ohio 44115\\
USA}
\email{jfoprea@gmail.com}

\begin{abstract}
For a graph $\Gamma$, let $K(H_{\Gamma},1)$ denote the Eilenberg-Mac Lane space
associated to the right-angled Artin (RAA) group $H_{\Gamma}$ defined by~$\Gamma$. We use the relationship between the combinatorics of $\Gamma$ and the topological complexity of $K(H_{\Gamma},1)$ to explain, and generalize to the higher TC realm, Dranishnikov's observation that the topological complexity of a covering space can be larger than that of the base space. In the process, for any positive integer $n$, we construct a graph $\mathcal{O}_n$ whose TC-generating function has polynomial numerator of degree $n$. Additionally, motivated by the fact that $K(H_{\Gamma},1)$ can be realized as a polyhedral product, we study the LS category and topological complexity of more general polyhedral product spaces. In particular, we use the concept of a strong axial map in order to give an estimate, sharp in a number of cases, of the topological complexity of a polyhedral product whose factors are real projective spaces. Our estimate
exhibits a mixed cat-TC phenomenon not present in the case of RAA groups.
\end{abstract}



\maketitle

\section{Main results}\label{sec:preliminaries}
The $\TC$ generating function of a space $X$, introduced in~\cite{FO}, is defined in terms of all the higher topological complexities $\TC_r(X)$:
$$\F_X (x) \, = \, \sum_{r=1}^\infty \TC_{r+1}(X)\cdot x^r.$$
We focus on the integral power series $P_X(x)=(1-x)^{-2}\cdot\F_X (x)$. Note that $P_X(0)=0$, whereas $P_X(x)$ vanishes if and only if $X$ is contractible. Computational evidence suggests that $P_X(x)$ would be a polynomial when $X$ is a finite CW complex. The following result, given in \cite{FO} without proof, restates the question as an eventual linear growth of the higher topological complexities of $X$. The proof is elementary (see for instance \cite[Lemma~1]{FKS}).

\begin{theorem}\label{thm:genequiv}
\blue{Let $X$ be a finite CW complex. $P_X(x)$ is a polynomial if and only if there is an integer $K_X$ satisfying $\TC_{r+1}(X) = \TC_r(X) + K_X$ for all $r$ large enough.
If these conditions hold, then:
\begin{itemize}
\item[{\it(a)}] $K_X$ is determined by $P_X$, $K_X=P_X(1)$. Conversely, setting $\TC_1(X):=0$ and 
$e:=\text{min}\{s\geq1\colon \TC_{r+1}(X) = \TC_r(X) + K_X, \text{ for all } r\geq s\}$, we have 
$$ P_X(x)=(1-x)^2\left(\sum_{i=2}^{e-1}\TC_i(X)\cdot x^{i-1}\right)+\TC_e(X)\cdot x^{e-1}(1-x)+K_X\cdot x^e.$$
\item[{\it(b)}] $P_X'(1)=eK_X-\TC_e(X)$.
\item[{\it(c)}] $\deg(P_X)=e$ provided $X\not\simeq\star.$
\end{itemize}
}
\end{theorem}

All generating functions arising in \cite{FO} have polynomial numerators $P_X(x)$ of degree one or two. To get a much richer source of examples, we consider the Eilenberg-Mac Lane space $K(H_{\Gamma},1)$ associated to a right-angled Artin (RAA) group $H_\Gamma$. Here $\Gamma$ is a simplicial graph (i.e., containing neither loops nor repeated edges), and $H_\Gamma$ is generated by $V_\Gamma$, the vertex set of~$\Gamma$, with relations $uv=vu$ whenever $u$ and $v$ are adjacent vertices in~$\Gamma$. It is well-known that, if $c(\Gamma)$ denotes the size of the largest clique in $\Gamma$, and $\cd(H_\Gamma)$ stands for the cohomological dimension of the group $H_{\Gamma}$, then
\begin{equation}\label{cdccatem}
\cd(H_\Gamma)=c(\Gamma)=\red{\cat(K(H_\Gamma,1)).}
\end{equation}
This essentially follows from the construction of the finite Salvetti complex (which is a $K(H_\Gamma,1)$) associated to $\Gamma$. We abuse notation by simply writing $\mathcal{F}_\Gamma(x)$, $P_\Gamma(x)$, $K_\Gamma$, $\TC_r(\Gamma)$ or even $\cat(\Gamma)$ when referring to the corresponding objects for $X=K(H_{\Gamma},1)$.  The following was proved in~\cite{FO}:

\begin{theorem}\label{gen}
$P_\Gamma(x)$ is an integer polynomial with $\deg(P_\Gamma(x))\leq|V_\Gamma|$, and $K_\Gamma =c(\Gamma)$ (recall~(\ref{cdccatem})). Thus $\tc_r(\Gamma) - \tc_n(\Gamma) = (r-n)c(\Gamma)$ for $r\geq|V_\Gamma|$.
\end{theorem}

The final equality in Theorem~\ref{gen} refines the general estimate $\tc_r - \tc_{r-1} \leq 2\cat$ coming from~\cite[Proposition~3.7]{BGRT}. However, the estimate $\deg(P_\Gamma(x))\leq V(\Gamma)$ is rather coarse. For instance, as noted in~\cite[Examples~8.6 and 8.7]{FO}, $P_\Gamma(x) = x(2-x)$ if $\Gamma$ has no edges, whereas $P_\Gamma(x)=nx$ if $\Gamma$ is the complete graph on $n$ vertices. In both cases we get polynomials of degree at most 2. More generally, $P_\Gamma(x)=kx+\ell x(1-x)$ if $\Gamma$ is the disjoint union of two cliques of cardinalities $k$ and $\ell$, with $k\geq\ell$ (as follows from Theorem~\ref{thm:genequiv} and~(\ref{wedgeoftori}) in Section~\ref{sec:examples}).

Our first main result contrasts with the three previous examples:
\begin{theorem}\label{thm:bigdeg}
There is a family $\{\mathcal{O}_n\}$ of graphs with $\deg(P_{\mathcal{O}_n}(x))=n$.
\end{theorem}

Beyond this, we shall see that $\TC$-generating functions reflect properties related to a positive conclusion in the single-space Ganea conjecture (Section~\ref{sec:sequence}).

\smallskip
We also use the graph theoretic viewpoint to explain a topological complexity phenomenon that was thought to be anomalous. Namely, Dranishnikov has shown in~\cite{Dra} that the topological complexity of a covering space can be greater than that of the base space. Our explanation of such a fact is done through the double $D_v(\Gamma)$ of a graph~$\Gamma$ around a vertex $v$ (the construction is reviewed in Section~\ref{sec:examples}). A key fact we need comes from \cite{KK}: There is an inclusion of RAA groups $H_{D_v(\Gamma)} \leq H_\Gamma$.
Thus $K(H_{D_v(\Gamma)},1)$ covers $K(H_\Gamma,1)$. In our second main result we give conditions that generalize Dranishnikov's observation at the $r$-th topological complexity level:

\begin{theorem}\label{thm:TCdouble}
If no vertex of $\mathrm{St}(v)$, the star of $v$ in $\Gamma$, belongs to a clique $C$ of $\,\Gamma$ of maximal cardinality, then $\TC_r(D_v(\Gamma))=r\cdot \cd(H_{D_v(\Gamma)}) = r \cdot \cd(H_\Gamma) = \red{r\cdot |C|}$, for all $r\geq 2$. In particular, $P_{D_v(\Gamma)}(x)=|C|\cdot x(2-x)$. On the other hand, if no vertex of $\mathrm{St}(v)$ is common to a pair of cliques $C_1$ and $C_2$ of $\,\Gamma$ with $|C_1|\geq|C_2|$ and $(r-1)|C_1|+|C_2|> \TC_r(\Gamma)$, then $\TC_{\red{r}}(D_v(\Gamma)) > \TC_{\red{r}}(\Gamma).$
\end{theorem}

Explicit examples satisfying the last conditions in Theorem~\ref{thm:TCdouble} (including Dranishnikov's double cover) are given in Section~\ref{sec:examples}.
Motivated by the fact that the Eilenberg-Mac Lane spaces associated to RAA groups may be realized as polyhedral products, we shall consider the LS category and higher topological complexities of general polyhedral product spaces $\und{X}^K$ defined in terms of a family $\und{X}=\{(X_i,*)\}_{i=1}^m$ of based spaces and a simplicial complex $K$ on vertices $[m]:=\{1,2,\ldots,m\}$ (explicit definitions are given in Subsection~\ref{subsec:polycat}). We start with the following generalization of \cite[Proposition~4]{FT}:

\begin{theorem}\label{thm:polycat}
If the family $\{(X_i,*)\}_{i=1}^m$ \red{is LS-logarithmic, meaning that}
$$\cat(X_{i_1} \times \cdots \times X_{i_k})=\cat(X_{i_1})+\cdots+\cat(X_{i_k})$$
\red{is satisfied} for any subcollection $X_{i_1}, \ldots, X_{i_k}$ \red{with $i_r\neq i_s$ for $r\neq s$, then}
$$\cat(\und{X}^K) = \max_{\sigma=[i_1,\ldots,i_n]}\left\{\cat(X_{i_1})+\cdots + \cat(X_{i_n})\rule{0pt}{4mm}\right\},$$
where the maximum is taken over all simplices $\sigma$ of $K$.
\end{theorem}

Regarding TC, we prove:
\begin{theorem}\label{thm:polyTC}
Assume the following three conditions hold:
\begin{enumerate}
\item $\TC_{\red{r}}(X_i)=\red{r}\,\cat(X_i)$ for all $X_i$;
\item \red{the family $\{(X_i,*)\}_{i=1}^m$ is LS-logarithmic;}
\item \red{the family $\{(X_i,*)\}_{i=1}^m$ is $\TC_r$-logarithmic, in the sense that} $$\TC_{\red{r}}(X_{i_1} \times \cdots \times X_{i_k})=\TC_{\red{r}}(X_{i_1})+\cdots+\TC_{\red{r}}(X_{i_k})$$ \red{is satisfied} for any subcollection $X_{i_1}, \ldots, X_{i_k}$ \red{with $i_r\neq i_s$ for $r\neq s$}.
\end{enumerate}
Then
$$\TC_{\red{r}}(\und{X}^K) = \red{r \,\cat(\und{X}^K)} = \max_{\sigma=[i_1,\ldots,i_n]}\left\{\TC_{\red{r}}(X_{i_1})+\cdots + \TC_{\red{r}}(X_{i_n})\rule{0mm}{4mm}\right\},$$
where the maximum is taken over all simplices $\sigma$ of $K$.
\end{theorem}

We shall see in Section~\ref{sec:prodsum} that conditions (2) and (3) in Theorem~\ref{thm:polyTC} are satisfied when each $X_i$ is a $K(H_{\Gamma_i},1)$, while condition (1) depends on the structure of the graphs $\Gamma_i$. 
Lastly, we use families $\und{P}=\{(P^{n_i},\ast)\}$ of real projective spaces to exhibit $\TC$-phenomena that do not seem to have a counterpart in the RAA group realm ---a combined maximum taken over sums of both cat and TC terms:

\begin{theorem}\label{Upper bound TC of Z(Pn,K)}
$\TC(\und{P}^K)$ is estimated from above by
$$
\TC(\und{P}^K)\leq\max\left\{\sum_{i\in \sigma_{1}\bigtriangleup\sigma_{2}}n_i+\sum_{i\in \sigma_1\cap\sigma_{2}}\TC(\RP^{n_i}):\sigma_{1},\sigma_{2}\in K\right\},
$$
where $\sigma_{1}\bigtriangleup\sigma_{2}=(\sigma_{1}\setminus\sigma_{2})\cup(\sigma_{2}\setminus\sigma_{1})$, the symmetric difference of $\sigma_1$ and $\sigma_2$.
\end{theorem}

The following lower estimate for $\TC(\und{P}^K)$, with the flavor of Theorem~\ref{Upper bound TC of Z(Pn,K)}, is closely related to the fact that the inequality in Theorem~\ref{Upper bound TC of Z(Pn,K)} is in fact sharp in a number of cases (Section~\ref{sec:rpss}):

\begin{proposition}\label{mixedlow}
$\TC(\und{P}^K)$ is estimated from below by
$$
\TC(\und{P}^K)\geq\max\left\{\sum_{i\in \sigma_{1}\bigtriangleup\sigma_{2}}n_i+\sum_{i\in \sigma_{1}\cap\sigma_{2}}\zcl(P^{n_i}): \sigma_{1}, \sigma_{2}\in K\right\},
$$
where $\zcl$ stands for zero-divisors cup-length with mod 2 coefficients.
\end{proposition}

Our proof of Theorem~\ref{Upper bound TC of Z(Pn,K)} does not seem to produce interesting results in the context of higher topological complexity. As an alternative, we prove:

\begin{theorem}\label{higherproj}
If all $n_i$ are even, then for large enough $r$,
$$
\TC_r(\und{P}^K)=r\,\cat(\und{P}^K)=\max_{\sigma=[i_1,\ldots,i_\ell]}\left\{\TC_{\red{r}}(P^{n_{i_1}})+\cdots + \TC_r(P^{n_{i_\ell}})\rule{0mm}{4mm}\right\},
$$
where the maximum is taken over all simplices $\sigma$ of $K$.
\end{theorem}

\section{Background}\label{sec:request}
Topological complexity $\TC(X)$, a homotopy invariant of a path-connected space $X$, was introduced in \cite{Far} motivated by topological aspects of the motion planning problem in robotics. 
Topological complexity is a close relative of the Lusternik--Schnirelmann (LS) category $\cat(X)$, and both are special cases of Schwarz' sectional category of a map. We review basic definitions, referring the reader to the monographs \cite{CLOT, Inv, Sv,Rud} for further details.

The sectional category of a fibration $p \colon E \to B$, denoted by $\secat(p)$, is the smallest $n$ for which there is an open covering $\{ U_0, \ldots, U_n \}$ of $B$ by $n+1$ open sets, on each of which there is a continuous local section $s_i \colon U_i \to E$ of  $p$, i.e., $p\circ s_i = j_i \colon U_i \hookrightarrow B$. We agree to set  $\secat(p)=\infty$ if the corresponding finite coverings fail to exist. The \emph{Lusternik-Schnirelmann category} is then defined to be $\cat(X)=\secat(PX \to X)$, where $PX$ is the space of \emph{based} paths in $X$ and the map $PX \to X$ takes a path to its endpoint.

The \emph{topological complexities} of a space $X$, denoted $\tc_r(X)$ for $r\geq 2$ are defined by $\tc_r(X)=\secat(p_r)$ where $p_r \colon X^I\to X^r$ is the fibration given by
\begin{eqnarray}
p_r(\alpha) = \left(\alpha(0), \alpha\left(\frac{1}{r-1}\right), \alpha\left(\frac{2}{r-1}\right), \dots, \alpha\left(\frac{r-2}{r-1}\right), \alpha(1)\right),
\end{eqnarray}
\red{for $\alpha\in X^I.$} Here $X^I$ denotes the space of free paths $\alpha: I=[0, 1]\to X$ and $X^r=X\times \dots\times X$,  the Cartesian product of $r$ copies of $X$.
When $r>2$, we say that the $\tc_r(X)$ are \emph{higher} or \emph{sequential} topological complexities of $X$ and when $r=2$, we say that $\tc(X)$ is simply \emph{the}
topological complexitiy of $X$.


Just as LS category is very difficult to compute, so also is topological complexity. Indeed, it is usually the case for both invariants that
lower and upper bounds are derived. The fundamental such bounds are:

\begin{theorem}\label{thm:bounds}
For any space $X$, $\cat(X^{r-1}) \leq \TC_r(X) \leq \cat(X^r)$.
\end{theorem}

When $X=K(\pi,1)$ is aspherical the topological complexities $\TC_r(X)$ depend only on $\pi$ and we may write $\TC_r(X)=\TC_r(\pi)$.
It is easy to see (using the Eilenberg - Ganea theorem \cite{EG}) that  $\TC_r(\pi)$ is finite if and only if there exists a finite dimensional $K(\pi,1)$.
Various results about $\TC_r(\pi)$ have been obtained (e.g.~\cite{FGLO,FM,GLO}), but there is no definitive result at this time.

\subsection{Higher topological complexity of RAA groups}\label{subsec:highRAA}
We use the notation set forth in Section~\ref{sec:preliminaries} regarding RAA groups. The higher topological complexity $\tc_r(H_\Gamma)$ was computed in~\cite{GGY} in terms of the structure of cliques in $\Gamma$. The answer is spelled out in Theorem \ref{GGY} below in terms of the simplified expression in~\cite{FO}.

\begin{definition}\label{def:zr}
For an integer $r\ge 2$ and a graph $\Gamma$, we define the number $z_r(\Gamma)$ as the maximal total cardinality $\sum_{i=1}^r |C_i|$ of $r$ cliques $C_1, \dots, C_r$ of $\,\Gamma$ with empty total intersection, $\cap_{i=1}^r C_i=\varnothing$.
\end{definition}

The maximum sum \red{in Definition~\ref{def:zr}} giving the value of $z_r(\Gamma)$ might not be realized by cliques whose cardinality is maximal.


\begin{theorem}[\cite{FO},\cite{GGY}]\label{GGY} For any graph $\Gamma$ and $r\geq2$, $\tc_r(H_\Gamma) = z_r(\Gamma)$.
\end{theorem}

\section{$\tc$-generating function and LS-logarithmicity}\label{sec:sequence}
A number of examples of finite CW complexes $X$ were given in \cite{FO} for which $P_X(x)$ is a polynomial satisfying in fact $P_X(1)=\cat(X)$ (cf.~Theorem~\ref{thm:genequiv}(a)). Real projective spaces $P^n$ give additional such examples provided $n$ is a 2-power. Indeed,~\cite[Proposition~6.2]{CGGGL} implies $P_{P^{2^e}}(x)=(2^{e+1}-1)x -2(2^{e-1}-1)x^2-x^3$, so $P_{P^{2^e}}(1)=2^e=\cat(P^{2^e})$. Note that all these examples deal with LS-logarithmic spaces, i.e., spaces $X$ satisfying\footnote{Of course, we always have $\cat(X^r)\leq r\cdot \cat(X)$.} $\cat(X^r)=r\cdot \cat(X)$ for any~$r$. In fact, in this short section we prove that whether a space is LS-logarithmic has a great influence on the structure of its $\TC$-generating function.

\begin{proposition}\label{cor:catdiff}
If $X$ is LS-logarithmic and $\TC_{r+1}(X)-\TC_r(X)=K$
for all $r \geq n$, then $K = \cat(X)$. Consequently $P_X(1)=\cat(X)$.
\end{proposition}

\begin{proof} 
For $r\geq n$, Theorem~\ref{thm:bounds} and the LS-logarithmicity hypothesis give
\begin{equation}\label{fid}
\overunderbraces{&\br{2}{K-s_{r+1}}& &\br{2}{K-s_r}}{&(r+1)\cat(X) \geq &\TC_{r+1}(X)& \geq &r\cat(X)& \geq \TC_r(X)}{&&\br{3}{s_r}},
\end{equation}
where $s_r:=\TC_{r+1}(X)-r\cat(X)$, so that $r\cat(X) - \TC_r(X) = K-s_r$, since $\TC_{r+1}(X)-\TC_r(X)=K$. Note that
\begin{align*}
s_r - s_{r+1} & = \left(\rule{0mm}{4mm}s_r-\TC_{r+1}(X)\right) + \left(\rule{0mm}{4mm}\TC_{r+1}(X)-s_{r+1}\right)\\
& =- r\cat(X)+(r+1)\cat(X)-K = \cat(X)-K.
\end{align*}
The latter expression is independent of $r$, so the relation $K\neq\cat(X)$ would imply that the sequence of integers $s_n,s_{n+1},\ldots$ is either strictly increasing or strictly decreasing, which is impossible as $0\leq s_r\leq\cat(X)$ for $r\geq n$.
\end{proof}

\begin{example}
\red{Theorem~5.7 in \cite{CGGGL} shows $\TC_r(P^m)=rm$ for $m$ even and~$r$ (moderately) large. Since $P^m$ is LS-logarithmic, Proposition~\ref{cor:catdiff} implies $P_{P^m}(1)=m$, even though the polynomial $P_{P^m}(x)$ is unknown. (Recall from~\cite{FTY} that $\TC(P^m)$ is the Euclidean immersion dimension of $P^m$.)}
\end{example}

If the LS-logarithmicity hypothesis in Proposition~\ref{cor:catdiff} is removed, the conclusion $P_X(1)=\cat(X)$ may be lost. Indeed, a remarkable 3-cell complex $X$ is constructed in \cite{Sta} with the property that $\cat(X \times X)=\cat(X)=2$. In \cite{FKS}, the higher topological complexities of $X$ were shown to be $\TC_r(X)=r=\cat(\red{X^r})$ \red{($r\geq2$)}. In view of Theorem~\ref{thm:genequiv}, this yields $P_X(1)=1<\cat(X)$. It should be noted that, in this example, $P_X(1)=\cupl(X)$, the cup-length of $X$, a situation that, to the best of our knowledge, holds in all cases where all $\TC_r(X)$ are known. It is tempting to insert this clause in the modified form of the rationality conjecture at the end of~\cite{FKS}.

\section{$\TC_r$ and coverings of Right-Angled Artin Groups}\label{sec:examples}
RAA groups can be used to gain insight into various phenomena in the study of topological complexity. For instance, in \cite{Rud2}, it was shown that
$\TC(T^k \vee T^\ell)=k+\ell$ and $\cat(T^k \vee T^\ell)=k$
provided $k\geq\ell$, where $T^r=S^1\times\cdots\times S^1$ is the $r$-dimensional torus. This showed that TC could lie anywhere in the range prescribed by Theorem~\ref{thm:bounds}. Such a fact has a simple explanation, and can be generalized to higher-$\TC$, 
simply by noticing that $T^k\vee T^\ell=K(\Z^k * \Z^\ell,1)$ and that $\Z^k *\Z^\ell$ is the right-angled Artin group associated to the graph consisting of a $k$-clique and a disjoint $\ell$-clique. For then Theorems~\ref{thm:polycat} and~\ref{GGY} yield, for any $r\geq2$,
\begin{equation}\label{wedgeoftori}
\TC_r(T^k\vee T^\ell)=(r-1)k+\ell\ \ {\rm and }\ \ \cat((T^k\vee T^\ell)^{r-1})=(r-1)k.
\end{equation}
The above philosophy is applied in this section to explain on combinatorial grounds, and generalize to the higher TC realm, Dranishnikov's example of a covering space $E\to B$ with $\TC(E)>\TC(B)$. In fact, from this perspective it is clear that such a phenomenon is more common than it was originally thought.

We use the notation set forth in Section~\ref{sec:preliminaries} regarding RAA groups. The double of a graph $\Gamma$ around one of its vertices $v$, denoted by $D_v(\Gamma)$, is formed by taking two copies of $\Gamma$ and identifying the copies of $\mathrm{St}(v)$ in each. Here, the star of $v$, $\mathrm{St}(v)$, is the induced subgraph of $\Gamma$ containing $v$ and all vertices connected to $v$ by an edge. Figures~\ref{fig:dranish} and~\ref{fig:graph2} illustrate the construction.

\begin{proof}[Proof of Theorem~\ref{thm:TCdouble}] 
\hspace{.3mm}By the hypothesis and the construction of $D_v(\Gamma)$, we see that $D_v(\Gamma)$ contains two disjoint cliques \red{of maximal cardinality,} $C$ and $C'$ (where the prime denotes inclusion in the second copy of $\Gamma$, $\Gamma'$). Since no other clique has larger size, we get the first assertion of the theorem, i.e.,
$$\TC_r(D_v(\Gamma)) = |C'| + (r-1)|C| = r|C| = r\cdot \cd(H_\Gamma)=r\cdot \cd(H_{D_v(\Gamma)}),$$
while the expression for $P_{D_v(\Gamma)}(x)$ follows directly from Theorem~\ref{thm:genequiv}. On the other hand, for cliques $C_i$ as in the last assertion of the theorem, the condition \red{$C_1 \cap C_2 \cap V_{\mathrm{St}(v)}=\varnothing$} implies $C_1 \cap C_2' = \varnothing$ (where, again, primes denote copies in the copy of $\Gamma$, $\Gamma' \subset D_v(\Gamma)$). So
$$\TC_{\red{r}}(D_v(\Gamma)) \geq \red{(r-1)}|C_1|+|C_2'| = \red{(r-1)|C_1|+|C_2|} > \TC_{\red{r}}(\Gamma),$$ which completes the proof.
\end{proof}

Dranishnikov's example in~\cite[Theorem~3.8(a)]{Dra} of a double covering space $p\colon E\simeq T^2\vee S^1\vee T^2\to B=S^1 \vee T^2$ with $\TC(E)>\TC(B)$ arises from the graph $\Gamma$ \red{in Figure~\ref{fig:dranish}} given by a $2$-clique $C$ and a disjoint vertex $v$. In this case Theorem~\ref{thm:TCdouble} is applied with  $C_1=C_2=C$. Note that $P_B(x)=3x-x^2$ and $P_{E}(x)=4x-2x^2$, both of degree 2.

\begin{figure}[h!]
\centering
\begin{tikzpicture}
\draw[fill=black] (0,0) circle (3pt);
\draw[fill=black] (1,0) circle (3pt);
\draw[fill=black] (1,1) circle (3pt);
\node at (0,-0.5) {$v$};
\node at (1,-0.5) {$C$};
\node at (-1,0.5) {$\Gamma$};
\draw[thick] (1,0) -- (1,1);
\end{tikzpicture}
\hspace{.7in}
\begin{tikzpicture}
\draw[fill=black] (0,0) circle (3pt);
\draw[fill=black] (1,0) circle (3pt);
\draw[fill=black] (1,1) circle (3pt);
\draw[fill=black] (-1,0) circle (3pt);
\draw[fill=black] (-1,1) circle (3pt);
\node at (1,-0.5) {$C$};
\node at (-1,-0.5) {$C'$};
\node at (0,-0.5) {$v$};
\node at (2.5,0.5) {$D_v(\Gamma)$};
\draw[thick] (1,0) -- (1,1);
\draw[thick] (-1,0) -- (-1,1);
\end{tikzpicture}
\caption{$\TC_r(\Gamma)=\red{2r-1}\hspace{.6mm}$ and $\,\TC_r(D_v(\Gamma))=\red{2r}$.}\label{fig:dranish}
\end{figure}

Wedges in Dranishnikov's example correspond to disconnected graphs (so that the associated RAA groups decompose as free products). But this is an unnecessary restriction. Consider for instance the situation in \figref{fig:graph2}, where
$$\TC_r(D_v(\Gamma))=3r-1 > 3r-2 = \TC_r(\Gamma),$$
for all $r\geq 2$ in view of Theorem~\ref{GGY}. Note that Theorem~\ref{thm:TCdouble} is used with $C_2=C_1\setminus\mbox{St}(v)$. This time $P_{\Gamma}(x)=4x-x^2$ and $P_{D_v(\Gamma)}(x)=5x-2x^2$.

\smallskip
More interesting is the situation for the graph $\Gamma$ shown in \figref{fig:graph1}, for only the regular topological complexity (i.e., $\TC_2$) of the corresponding covering space is strictly larger than that of the base space. We leave it to the reader to show that, while $\TC(D_v(\Gamma)) = 8$ and $\tc(\Gamma)=7$, for $r>2$, we have $\TC_r(D_v(\Gamma))=\TC_r(\Gamma)$.
This phenomenon reflects the fact that the hypotheses in the second half of Theorem~\ref{thm:TCdouble} fail for $C_1$ and $C_2$ if $r\geq3$. Note that in this case $P_{\Gamma}(x)=7x-x^2-x^3$ while $P_{D_v(\Gamma)}(x)=8x-3x^2$, a smaller degree polynomial.


\begin{figure}[h!]
\centering
\begin{tikzpicture}
\draw[fill=black] (0,0) circle (3pt);
\draw[fill=black] (0,1) circle (3pt);
\draw[fill=black] (2,1) circle (3pt);
\draw[fill=black] (1,2) circle (3pt);
\node at (0,-0.5) {$v$};
\node at (1,1.35) {$C_1$};
\node at (-1,1) {$\Gamma$};
\draw[thick] (0,0) -- (0,1) -- (2,1) -- (1,2) -- (0,1);
\end{tikzpicture}
\hspace{.5in}
\begin{tikzpicture}
\draw[fill=black] (0,0) circle (3pt);
\draw[fill=black] (0,1) circle (3pt);
\draw[fill=black] (-2,1) circle (3pt);
\draw[fill=black] (-1,2) circle (3pt);
\draw[fill=black] (2,1) circle (3pt);
\draw[fill=black] (1,2) circle (3pt);
\node at (0,-0.5) {$v$};
\node at (3.3,1) {$D_v(\Gamma)$};
\draw[thick] (0,0) -- (0,1) -- (2,1) -- (1,2) -- (0,1) -- (-1,2) -- (-2,1) -- (0,1);
\end{tikzpicture}
\caption{$\TC_r(\Gamma)=\red{3r-2}$ and $\TC_r(D_v(\Gamma))=\red{3r-1.}$}\label{fig:graph2}
\end{figure}

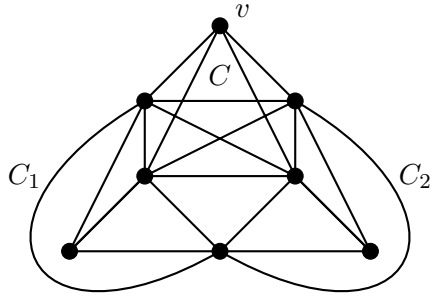
\begin{figure}[h!]
\centering
\begin{tikzpicture}
\draw[fill=black] (0,0) circle (3pt);
\draw[fill=black] (2,0) circle (3pt);
\draw[fill=black] (-2,0) circle (3pt);
\draw[fill=black] (-1,1) circle (3pt);
\draw[fill=black] (1,1) circle (3pt);
\draw[fill=black] (-1,2) circle (3pt);
\draw[fill=black] (1,2) circle (3pt);
\draw[fill=black] (0,3) circle (3pt);
\node at (.3,3.2) {$v$};
\node at (-2.6,1) {$C_1$};
\node at (2.6,1) {$C_2$};
\node at (0,2.35) {$C$};
\draw[thick] (0,0) -- (-2,0) -- (-1,1) -- (0,0);
\draw[thick] (0,0) -- (2,0) -- (1,1) -- (0,0);
\draw[thick] (0,0) to[out=-150,in=-150, distance=3cm ] (-1,2);
\draw[thick] (-1,2) -- (-2,0) -- (-1,1) -- (1,1) -- (2,0) -- (1,2) -- (1,1) -- (-1,2) -- (1,2) -- (-1,1) -- (0,3) -- (1,1) ;
\draw[thick] (0,3) -- (-1,2) -- (-1,1);
\draw[thick] (0,3) -- (1,2);
\draw[thick] (0,0) to[out=-30,in=-30, distance=3cm ] (1,2);
\end{tikzpicture}
\caption{$\TC(H_\Gamma)=7$ and $\TC(H_{D_v(\Gamma)})=8.$}\label{fig:graph1}
\end{figure}

In light of the above, a $\TC$-generating-function anomalous situation would be given by a covering $E\to B$ with $\deg(P_E(x))>\deg(P_B(x))$.

\section{$P_X$ of Degree Greater than \red{$3$}}\label{sec:highdegree}
The examples of $\TC$-generating functions above all have $\deg(P_X(x))\leq3$. We now construct examples with arbitrary $\deg(P_X(x))$.

\begin{proposition}\label{prop:formula}
If $z_r(\Gamma)$ can be realized by a sequence of cliques $C_1,\ldots,C_r$ \red{of $\Gamma$,} where $\cap_{s=1}^k C_{i_s} = \varnothing$ for some subsequence of length $k$, then
$$z_r(\Gamma) = z_k(\Gamma) + (r-k)c(\Gamma).$$
\end{proposition}

\begin{proof}
Without loss of generality suppose that $C_1\cap\cdots\cap C_k=\varnothing$. By the definition of $z_k:=z_k(\Gamma)$, there are cliques
$K_1,\ldots,K_k$ with empty total intersection and $z_k=\sum_{i=1}^k |K_i|$. Moreover, since $z_k$ maximizes such sums, we must have
$$\sum_{i=1}^k |C_i| \leq \sum_{i=1}^k |K_i|.$$
However, if the inequality were strict, then we would have
$$z_r = \sum_{i=1}^r |C_i| < \sum_{i=1}^k |K_i| + \sum_{i=k+1}^r |C_i|,$$
which would contradict the maximality of $z_r$ since $K_1,\ldots,K_k,C_{k+1},\ldots,C_r$ has empty total intersection. Thus, for any $C_1,\ldots,C_k$ with empty total intersection in a sequence realizing $z_r$, we must have $\sum_{i=1}^k |C_i| = z_k$. But now we see that the sequence $C_1, \ldots,C_k,C_\Gamma,\ldots,C_\Gamma$ (where $C_\Gamma$, a clique of maximal size, is taken $(r-k)$-times) also forms a clique sequence with empty total intersection. Since $|C_j| \leq |C_\Gamma|$ for all $j=k+1,\ldots,r$, this sequence must be one that
realizes the maximum $z_r$. But then we have
$$z_r = \sum_{i=1}^k |C_i| + (r-k) |C_\Gamma| = z_k + (r-k) c(\Gamma),$$
as asserted.
\end{proof}

The following example paves the way for the main result in this section.
\vspace{.1in}

\begin{figure}[h]
\centering
\begin{tikzpicture}
\draw[fill=black] (0,0) circle (3pt);
\draw[fill=black] (2,2) circle (3pt);
\draw[fill=black] (-2,2) circle (3pt);
\draw[fill=black] (-1,1) circle (3pt);
\draw[fill=black] (1,1) circle (3pt);
\draw[fill=black] (0,2) circle (3pt);
\node at (-1,1.6) {$C_1$};
\node at (1,1.6) {$C_2$};
\node at (0,.6) {$C_3$};
\draw[thick] (0,0) -- (-1,1) -- (-2,2) -- (0,2) -- (-1,1) -- (1,1) -- (0,2) -- (2,2) -- (1,1) -- (0,0);
\end{tikzpicture}
\caption{$\mathcal{O}_3$}\label{fig:o3}
\end{figure}
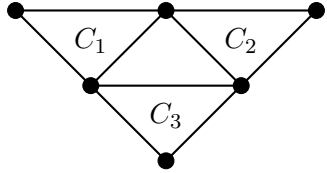

\begin{example}{\rm
The graph $\mathcal{O}_3$ in \figref{fig:o3}, with $6$ vertices and $9$ edges, has a central triangle with each side being the base of another triangle. Then, because we require cliques that are disjoint, we have $z_2(\mathcal{O}_3)=5$. However, for $z_3(\mathcal{O}_3)$, we only require empty total intersection, so we can take $C_1,C_2,C_3$ with each $C_j$ being one of the side triangles. This yields the maximum possible $z_3(\mathcal{O}_3)=9$. Note that we do not have $z_3(\mathcal{O}_3)=z_2(\mathcal{O}_3)+c(\Gamma)$ in this case. This is because the cliques $C_j$ have $C_i \cap C_j \not = \varnothing$ for all choices of $i,j=1,2,3$, so the hypothesis of Proposition~\ref{prop:formula} never holds. However, because the triangles $C_j$ use up all the vertices of the graph, any longer sequence $C_1,\ldots,C_r$ can do no better than starting with the triangles. Therefore, we must have, for any $r\geq 3$,
$$z_r(\mathcal{O}_3) = z_3(\mathcal{O}_3) + (r-3) c(\Gamma) = 9 + (r-3)(3) = 3r$$
which is the largest possible. \red{As a consequence $P_{H_3}(x)=-x^3-x^2+5x$, a polynomial of} degree 3.
}\end{example}

More generally, let $K_n$ be the complete graph on $n$ vertices. This forms the central clique akin to the central triangle of $\mathcal{O}_3$. Now, there are $n$ subsets of $(n-1)$ vertices in $K_n$ and for each of these take the join with a single new vertex. The graph $\mathcal{O}_n$ consists of $K_n$ and the $n$ cliques $C_1,\ldots,C_n$ (also of size $n$) corresponding to each of these joins.

\begin{proof}[Proof of Theorem~\ref{thm:bigdeg}]
We compute the numbers $z_r:=z_r(\mathcal{O}_n)$. For $z_2$, we see that any two cliques $C_i$ and $C_j$ intersect in $n-2$ vertices, so we are forced to take any $C_j$ and then use the central vertex not in $C_j$ with one of the new vertices attached to it to get $z_2=n+2$. For $z_3$, as noted above, $C_i \cap C_j$ overlap in $n-2$ central vertices, so the two vertices not in the intersection can be used with a new vertex to form a $K_3$ clique. Hence, $z_3=2n+3$. For $z_4$, we see that any
$C_i \cap C_j \cap C_k$ consists of $n-3$ central vertices, so the three central vertices not in the intersection can be used with a new vertex (not in $C_i$, $C_j$ or $C_k$ of course) to obtain a $K_4$ clique. Hence, $z_4=3n+4$. This same argument works up to $z_{n-1}=(n-2)n+n-1$ since any $n-2$ of the $C_i$ cliques intersect in $n-(n-2)=2$ central vertices, thus allowing a $K_{n-1}$ clique to be formed. However, for $z_n$, we may use all the $C_i$ since $\cap_{i=1}^n C_i = \varnothing$, so $z_n=n^2$. Then, by Proposition \ref{prop:formula}, we have $z_r = z_n + (r-n)n=rn$ for $r>n$. The result now follows from part (c) in Theorem~\ref{thm:genequiv}.
\end{proof}

\section{Consequences of the LS-Logarithmic Condition}\label{sec:prodsum}
We have seen in Section~\ref{sec:sequence} that the LS-logarithmicity hypothesis entails interesting consequences. Here we give further instances of the influence of the LS-logarithmicity hypothesis. \red{In what follows all spaces are assumed to be normal.}

\subsection{Open Covers}\label{subsec:opencovs}
We recall certain techniques involving open covers of spaces. These techniques go back at least to \cite{Os,Sv} and more recently have been used in
\cite{Dra,FGLO2,OpSt}.

\begin{lemma}\label{lem:opencov}
Suppose $U_0, \ldots,U_k$ is an open covering of the space $X$. Then, for any integer $m \geq 0$, the open cover
may be extended (recursively) to an open cover $U_0, \ldots, U_{k+m}$ with the property that any $(k+1)$ subsets still cover $X$.
Moreover, the added subsets $U_{k+1},\ldots,U_{k+m}$ are formed by taking disjoint unions of subsets of the original
$U_0, \ldots,U_k$.
\end{lemma}

\begin{remark}\label{rem:catcov}
If the cover $U_0, \ldots,U_k$ is \emph{categorical} (i.e., each $U_j$ is contractible to a point in $X$), then the extended
cover $U_0, \ldots, U_{k+m}$ is also categorical by simply \emph{restricting the original contracting homotopies to open subsets}. Notice
that no problem is incurred with intersections since the extended cover consists of disjoint unions of subsets of the cover
$U_0, \ldots, U_{k}$.
\end{remark}

\red{An open} covering $U_0,\ldots,U_{k+m}$ having the property of Lemma~\ref{lem:opencov}, i.e., that any $(k+1)$ subsets cover $X$, is called a
$(k+1)$-cover. We then have:

\begin{lemma}\label{lem:plusone}
If $U_0,\ldots,U_{k+m}$ is a $(k+1)$-cover of $X$, then each point $x \in X$ is in at least $m+1$ of the subsets.
\end{lemma}

\begin{proof}
Suppose not. Then some $x \in X$ lies in only (at most) $m$ of the subsets, say $U_0,\ldots,U_{m-1}$. But then $x$ does not lie
in any of the $k+1$ subsets $U_m,\ldots,U_{k+m}$, and this contradicts the cover being a $(k+1)$-cover.
\end{proof}

\begin{proposition}\label{prop:coverings}
Let $X=\prod_{i=1}^n X_i$ be a product where each $X_i$ has an open covering $\red{\mathcal{U}_i}=\{{}_iU_j\}_{j=0}^{k_i}$. \red{For each $i$, extend $\mathcal{U}_i$ to a $(k_i+1)$-cover $\{{}_iU_j\}_{j=0}^{S}$ where $S=\sum_{i=1}^n k_i$ (this can be done in view of Lemma~\ref{lem:opencov}). Then the sets $W_j=\prod_{i=1}^n {}_iU_j$, with $j=0,1,\ldots, S$, form an open covering of $X$.}
\end{proposition}
\begin{proof}
Let $(x_1,x_2,\ldots,x_n)\in X$. The first
coordinate $x_1$ lies in at least $S-k_1+1$ subsets by \lemref{lem:plusone}, so without loss of generality take these to be
${}_1U_0,\ldots,{}_1U_{S-k_1}$. The point $x_2$ lies in at least $S-k_2+1$ of the subsets, so it must be in at least
$S-k_1 - k_2 +1$ of the sets ${}_2U_0,\ldots,{}_2U_{S-k_1}$. Again, we may take these to be ${}_2U_0,\ldots,{}_2U_{S-k_1-k_2}$.
Similar arguments say that $x_p$ lies in $S-\sum_{j=1}^p k_j +1$ of the sets ${}_pU_0,\ldots,{}_pU_{S-T}$, where $T=\sum_{j=1}^{p-1} k_j$, and we may take these to be ${}_pU_0,\ldots,{}_2U_{S-T-k_p}$. Finally, $x_n$ lies in at least
$S-S+1=1$ set which may be taken to be $_nU_0$. But then $(x_1,\ldots,x_n)\in W_0$. Hence $\{W_j\}$ covers.
\end{proof}

\subsection{LS Category of Products}\label{subsec:cat}
As a reference for our later use of open cover properties, we give an easy proof of the \red{standard} product inequality for LS category\red{:}

\begin{proposition}\label{prop:prodinequal}
For $X$ and $Y$ normal, $\cat(X \times Y) \leq \cat(X)+ \cat(Y)$.
\end{proposition}

\begin{proof}
Let $\cat(X)=k$ and $\cat(Y)=m$ be represented by respective open covers $U_0,\ldots,U_k$  and $V_0,\ldots,V_m$. Then
by \lemref{lem:opencov}, we may extend these covers to \red{categorical covers} $U_0,\ldots,U_{k+m}$ and $V_0,\ldots,V_{k+m}$ where the former
is a $(k+1)$-cover and the latter is an $(m+1)$-cover. Let
$$W_j=U_j \times V_j, \quad {\rm for}\ j=0,\ldots,k+m.$$
Clearly each $W_j$ is contractible in $X \times Y$ since we may use the product of the individual contracting homotopies in
$X$ and $Y$. By \propref{prop:coverings}, we know that $\{W_j\}$ covers $X \times Y$ and therefore, $\cat(X \times Y) \leq k+m$.
\end{proof}

\begin{remark}\label{rem:morespaces}
The exact same proof as above applies to show
$$\cat(\prod_{i=1}^m X_i) \leq \sum_{i=1}^m \cat(X_i).$$
Here, each open categorical cover is extended to one of the form
$U_0, \ldots, U_s$, where $s=\sum_{i=1}^m \cat(X_i)$.
\end{remark}

\subsection{LS Category of Polyhedral Products}\label{subsec:polycat}
Let $\{(X_i,A_i)\}_{i=1}^m$ be a family of pairs of spaces, and $K$ be a simplicial complex on vertices $[m]:=\{1,2,\ldots,m\}$. For a subset $\sigma \subseteq [m]$, let
$$X^\sigma = \prod_{i=1}^m Y_i, \quad {\rm where} \quad
Y_i =\begin{cases}
X_i & {\rm if}\ i \in \sigma; \\
A_i & {\rm if}\ i \not\in\sigma.
\end{cases}
$$
The polyhedral product $\mathcal{Z}(\{X_i,A_i\},K)$ is the union of all $X^\sigma$ with $\sigma$ running over the simplices\footnote{Note that $X^\tau\subseteq X^\sigma$ for $\tau \subseteq \sigma$. So, in forming the union~(\ref{launion}), it suffices to consider simplices $\sigma$ of $K$ not contained in any other simplex of $K$.} of $K$,
\begin{equation}\label{launion}
\mathcal{Z}(\{X_i,A_i\},K) = \cup_{\red{\sigma \in K}} X^\sigma.
\end{equation}
We are only interested in the case where all $A_i = *$, and use the notation $\und{X}^K=\mathcal{Z}(\{X_i,A_i\},K)$. If in addition all $X_i = X$, we simply write $X^K$. The connection to our earlier work, and motivation for the results given below is apparent once it is realized that any $K(H_\Gamma,1)$ is the polyhedral product $(S^1)^K$, where $K$ is the flag complex with $1$-skeleton given by $\Gamma$.

In \cite{FT}, the category of $X^K$ was computed to be $\cat(X)\cdot (1+ \dim(K))$ provided that $X$ is LS-logarithmic. Here we shall generalize this result to polyhedral products $\und{X}^K$. 
We start with the following result that is reminiscent of \cite[Theorem 5.3]{GGY}.

\begin{lemma}\label{lem:upbound}
Let $\{(X_i,*)\}_{i=1}^m$ be a family of based spaces and $K$ be a simplicial complex on vertices $[m]$. Then
$$
\cat(\und{X}^K) \leq \max_{\sigma=[i_1,\ldots,i_n]}\left\{\cat(X_{i_1})+\cdots + \cat(X_{i_n})\rule{0mm}{4mm}\right\},
$$
where the maximum is taken over all\,\footnote{Of course, in this type of expressions the maximum can be taken over all maximal simplices of $K$, i.e., simplices of $K$ that are not contained in any other simplex of $K$.} simplices $\sigma$ of $K$.
\end{lemma}

\begin{proof}
Let ${}_iU_0,\ldots,{}_iU_{k_i}$ denote an open cover of the space $X_i$ realizing $\cat(X_i)=k_i$. Fix homotopies
${}_iH_j$ contracting ${}_iU_j$ to a point in $X_i$. (These may be taken to be base pointed by \cite[Chapter 1]{CLOT}.)
Let $\sigma=[i_1,\ldots,i_n]$ be a maximal simplex giving a ``subproduct''
$$X^\sigma = \prod_{j=1}^n X_{i_j} \subseteq \und{X}^K.$$
Here an element of $X^\sigma$ is thought of as element of $\und{X}^K$ by completing with base points in the coordinates corresponding to vertices not in $\sigma$. By the proof of \propref{prop:prodinequal} and \remref{rem:morespaces}, we see that $X^\sigma$ is covered by $W_0^\sigma,\ldots,W_s^\sigma$, where $s=\sum_{j=1}^n k_{i_j}$ and $$W_j^\sigma = {}_{i_1}U_j \times \cdots \times {}_{i_n}U_j$$ for extended covers ${}_{i_j}U_0, \ldots, {}_{i_j}U_s$ for each space appearing in the product~$X^\sigma$. Note that, by \remref{rem:catcov}, the contracting homotopies for the extended covers are simply restrictions of the ${}_{i_j}H_r$. This is, in fact, the key point. If $\tau=[r_1,\ldots,r_p]$ is another maximal simplex of $K$, then we obtain
$$W_j^\tau = {}_{r_1}U_j \times \cdots \times {}_{r_p}U_j$$
for $j =0,\ldots, z=\sum_{i=1}^p k_{r_i}$. Here we use the fact that if $m >p$, then the extended open cover $U_0,\ldots,U_m$ of Lemma \ref{lem:opencov} is an extension of $U_0,\ldots,U_p$. Now, if $s > z$, then for $s \geq j > z$, let $W_j^\tau = \varnothing$. Then, because the open covers \emph{and} homotopies are fixed for each $X_i$, if $W_j^\sigma \cap W_j^\tau \not = \varnothing$, the homotopies and open subsets agree on the intersection. Hence, they can be pieced together to contract $W_j^\sigma \cup W_j^\tau$ to a point in $\und{X}^K$. This can then be done for all maximal simplices, so we define
$$W_j^K = \cup_\sigma W_j^\sigma,\ {\rm where}\ \sigma\ {\rm is\ maximal}.$$
Here, we have $j=0,\ldots, \max_{\sigma=[i_1,\ldots,i_n]}\{k_{i_1}+\cdots + k_{i_n}\}$, where the maximum is taken over maximal simplices. Notice that by taking unions, we only require the maximal simplex with the greatest number of open sets.
\end{proof}

\begin{proof}[Proof of Theorem~\ref{thm:polycat}]
By \lemref{lem:upbound}, we need only show that
\begin{equation}\label{noslb}
\max_{\sigma=[i_1,\ldots,i_n]}\{\cat(X_{i_1})+\cdots + \cat(X_{i_n})\}
\end{equation}
is a lower bound for $\cat(\und{X}^K)$. Let $\red{\sigma_0}$ be a simplex realizing~(\ref{noslb}). Since $X^{\red{\sigma_0}} \subseteq \und{X}^K$ is a subproduct, there is a retraction $\und{X}^K \to X^{\red{\sigma_0}}$ obtained by taking all
other coordinates to base points. But retracts have smaller category so we have $\cat(X^{\red{\sigma_0}}) \leq \cat(\und{X}^K)$. Now, in general,
we only have the inequality
$$\cat(X^{\red{\sigma_0}}) \leq \max_{\sigma=[i_1,\ldots,i_n]}\{\cat(X_{i_1})+\cdots + \cat(X_{i_n})\}$$
but because of the hypothesis, we obtain equality. Hence,
$$\red{\max_{\sigma=[i_1,\ldots,i_n]}\{\cat(X_{i_1})+\cdots + \cat(X_{i_n})\}=
\cat(X^{\red{\sigma_0}})} \leq \cat(\und{X}^K)$$
and we are done.
\end{proof}

\begin{example}\label{familyofgraphs}{\rm
\red{Let $\{\Gamma_i\}_{i=1}^m$ be a family of graphs whose sets of vertices are pairwise disjoint. The set of vertices of the graph join $\Gamma:=\Gamma_1\ast\cdots\ast \Gamma_m$ is the union of the vertices of all the graphs $\Gamma_i$ where,} besides the adjacencies in the individual graphs, each vertex of \red{$\Gamma_i$} is taken to be adjacent \red{in $\Gamma$} to each vertex of \red{$\Gamma_j$ whenever $i\neq j$.} Therefore, cliques \red{$C_i$ of $\Gamma_i$ ($1\leq i\leq m$),} form a new clique $\red{C_1\ast\cdots\ast C_m}$ (of size $\red{\sum_i |C_i|}$) of \red{$\Gamma$.} Hence, the size of the
largest clique in $\red{\Gamma}$ is the sum of the sizes of the largest cliques in the individual graphs,
\begin{equation}\label{cliquejoin}
\red{c(\Gamma)=\sum_{i=1}^mc(\Gamma_i).}
\end{equation}
In terms of the RAA groups, we have disjoint subgroups so that elements in different subgroups commute; hence, $H_\Gamma=H_{\Gamma_1} \times \cdots \times H_{\Gamma_m}$, \red{and (\ref{cliquejoin}) gives}
$$\cd(H_{\Gamma})=\cd(H_{\Gamma_1} \times \cdots \times H_{\Gamma_m}) = \sum_{i=1}^m \cd(H_{\Gamma_i}).$$ \red{In terms of classifying spaces, the above discussion shows that any collection $\{(K(H_{\Gamma_i},1),\ast)\}_{i=1}^m$ is LS-logarith\-mic. Although the associated polyhedral product $\und{\hspace{.3mm}\Gamma\hspace{.3mm}}^K:=\mathcal{Z}(\{K(H_{\Gamma_i},1),\ast\},K)$ classifies a RAA group, so its category is given by~(\ref{cdccatem}), Theorem \ref{thm:polycat} yields an explicit formula for $\cat(\und{\hspace{.3mm}\Gamma\hspace{.3mm}}^K)$ in terms of} the simplicial complex $K$ and the cliques of the graphs $\Gamma_i$.
}\end{example}


\subsection{TC of Polyhedral Products}\label{subsec:TCpoly}
For reference purposes, we \red{record the following easy consequence of Theorem~\ref{thm:bounds} and Lemma~\ref{lem:upbound}:}

\begin{corollary}\label{cor:TCupbound}
\red{For a family $\{(X_i,*)\}_{i=1}^m$ of based spaces and a simplicial complex $K$ on $[m]$ vertices, we have}
$$\TC_{\red{r}}(\und{X}^K) \leq \red{r}\left(\max_{\sigma=[i_1,\ldots,i_n]}\,\left\{\cat(X_{i_1})+\cdots + \cat(X_{i_n})\rule{0mm}{4mm}\right\}\right),$$
where the maximum is taken over all simplices $\sigma$ of $K$.
\end{corollary}


\begin{proof}[Proof of Theorem~\ref{thm:polyTC}]
We follow the argument in the proof of Theorem~\ref{thm:polycat}. By \corref{cor:TCupbound}, it suffices to show that
\begin{equation}\label{mimica}
\max_{\sigma=[i_1,\ldots,i_n]}\left\{\rule{0mm}{4mm}\TC_{\red{r}}(X_{i_1})+\cdots + \TC_{\red{r}}(X_{i_n})\right\}
\end{equation}
is a lower bound for $\TC_{\red{r}}(\und{X}^K)$. Let $\sigma_{\red{0}}$ be \red{a} simplex realizing~(\ref{mimica}). As in the previous proof, $X^{\sigma_{\red{0}}}$ is a retract of $\und{X}^K$, so we have $\TC_{\red{r}}(X^{\sigma_{\red{0}}}) \leq \TC_{\red{r}}(\und{X}^K)$. Now, in general, we only have the inequality $$\TC_{\red{r}}(X^{\sigma_{\red{0}}}) \leq \red{\max_{\sigma=[i_1,\ldots,i_n]}}
\left\{\rule{0mm}{4mm}\TC_{\red{r}}(X_{i_1})+\cdots + \TC_{\red{r}}(X_{i_n})\right\}$$ but because of the current hypotheses, we obtain equality. Hence, $$\max_{\sigma=[i_1,\ldots,i_n]}\left\{\rule{0mm}{4mm}\TC_{\red{r}}(X_{i_1})+\cdots + \TC_{\red{r}}(X_{i_n})\right\} = \TC_{\red{r}}(X^{\sigma_{\red{0}}}) \leq \TC_{\red{r}}(\und{X}^K)$$ and we are done.
\end{proof}

\noindent Say that $X$ is $\TC_{\red{r}}$-logarithmic if $\TC_{\red{r}}(X^s)=s\,\TC_{\red{r}}(X)$ for all $s \geq 2$.

\begin{corollary}\label{cor:polyTC}
For a family $\{(X,*)\}_{i=1}^m$ of $m$-copies of the same based space $X$ and a simplicial complex $K$ on $[m]$ vertices, if
$\TC_{\red{r}}(X)=\red{r}\,\cat(X)$ and $X$ is both LS- and $\TC_{\red{r}}$-logarithmic, then
$$\TC_{\red{r}}(X^K)=\TC_{\red{r}}(X)\cdot(1+\dim(K))=\red{r}\,\cat(X^K).$$
\end{corollary}

\begin{remark}
In \cite[Theorem 5.3]{GGY}, a result similar to \thmref{thm:polyTC} was proven using the hypotheses that, for each $i$,  $\TC_{\red{r}}(X_i)$ is given
by rational \red{$r$-th} zero-divisor cup-length and it equals twice the cone-length. Such hypotheses imply that the family $\{(X_i,*)\}_{i=1}^m$ is both LS- and 
$\TC_r$-logarithmic. In this sense, \thmref{thm:polyTC} is a generalization of that of \cite{GGY}.
\end{remark}

\red{Example~\ref{familyofgraphs} gives the LS-logarithmicity hypothesis in Theorem~\ref{thm:polyTC} for families of RAA groups; the $\TC$-logarithmicity hypothesis also holds:}

\begin{proposition}\label{prop:TCRAA}
\red{For any family $\{\Gamma_i\}_{i=1}^m$ of simplicial graphs,}
$$\TC_r({\Gamma_1} \ast \cdots \ast {\Gamma_m}) =\TC_r(\Gamma_1)+\cdots+\TC_r(\Gamma_m).$$
\end{proposition}

\begin{proof}
\red{It suffices to consider the case $m=2$. Let} $\Gamma$ and ${\Gamma'}$ be simplicial graphs \red{and choose} $C_1, \ldots, C_r$ and $D_1,\ldots,D_r$ cliques in $\Gamma$ and $\Gamma'$ respectively with $\cap_{i=1}^r C_i = \varnothing$
and $\cap_{i=1}^r D_i = \varnothing$ that realize $\TC_r(\Gamma)$ and  $\TC_r({\Gamma'})$. Let $E_i = C_i * D_i$ (where $*$
again denotes the graph join). Clearly, $\cap_{i=1}^r E_i = \varnothing$, \red{for} all $C_i$ are pairwise disjoint from all $D_j$ and the total
intersections of the $C_i$ and $D_i$ are individually empty. Then, using Theorem \ref{GGY}, we obtain
\begin{align*}
 \sum_{i=1}^r |C_i| + \sum_{i=1}^r |D_i| & = \sum_{i=1}^r |E_i| \leq \TC_r(\Gamma \ast {\Gamma'}) \\
& \leq \TC_r(\Gamma) + \TC({\Gamma'}) = \sum_{i=1}^r |C_i| + \sum_{i=1}^r |D_i|.
\end{align*}
Hence, all inequalities are equalities and we are done.
\end{proof}

The validity of the first hypothesis in Theorem \ref{thm:polyTC} for families of RAA groups depends on the actual structure of the graphs~$\Gamma_i$. \red{For instance,} if a graph $\Gamma$ has \red{two disjoint cliques of the maximal cardinality} (as is the case for $D_v(\Gamma)$ in the first part of Theorem \ref{thm:TCdouble}), then $\TC_{\red{r}}(\Gamma)=\red{r}\,\cd(\Gamma)$, in view of Theorem~\ref{GGY}. \red{Thus:}

\begin{corollary}\label{cor:RAApolyTC}
If $\{\Gamma_i\}_{i=1}^n$ is a family of simplicial graphs such that each $\Gamma_i$ has two disjoint \red{cliques of the maximal cardinality, then for any}
simplicial complex $K$,
$$\TC_{\red{r}}\left(\und{\hspace{.3mm}\Gamma\hspace{.3mm}}^K\right) = \max_{\sigma=[i_1,\ldots,i_n]}\left\{\TC_{\red{r}}({\Gamma_{i_1}})+\cdots + \TC_{\red{r}}({\Gamma_{i_n}})\rule{0mm}{4mm}\right\},$$
where the maximum is taken over all maximal simplices $\sigma$ of $K$. Here we follow the notation in Example~\ref{familyofgraphs}, in particular $\Gamma=\Gamma_1\ast\cdots\ast\Gamma_m$.
\end{corollary}

It would be interesting to know whether one or both of the logarithmicity hypotheses could be waived in Theorem~\ref{thm:polyTC}. The hypothesis $\TC_r=r\, \cat$ for the polyhedral product factors however, plays a crucial role even for $\TC_2$, as will be clear in the next and final section.

We close the section with a general lower estimate (used later in the paper) for the topological complexity of polyhedral products. We need:
\begin{theorem}[{\cite[Corollary 2.4]{GLO2}}]\label{thm:glo}
If $F \to E \to B$ is a fibration of connected CW complexes admiting a (homotopy) section, then $\cat(B \times F \times E^{r-2}) \leq \TC_r(E)$.
\end{theorem}

\begin{corollary}\label{cor:glo}
\red{Let} $\sigma_1,\ldots,\sigma_r$ \red{be an} $r$-tuple of simplices in $K$ such that at least two simplices are disjoint. Then, \red{for any family $\{(X_i,\ast)\}_{i=1}^m$,}
$$\cat(\red{{X}^{\sigma_1}\times {X}^{\sigma_2}\times\cdots\times{X}^{\sigma_r}}) \leq \TC_r(\red{\und{X}}^K).$$
\red{Here ${X}^{\sigma_i}$ stands for the product $\prod_{j\in\sigma_i}X_j$ (see the proof of Corollary~\ref{lem:upbound}).}
\end{corollary}

\begin{proof}
Without loss of generality take $\sigma_1 \cap \sigma_2 = \varnothing$. Consider the projection $p_1\colon \red{\und{X}}^K \to \red{{X}}^{\sigma_1}$ with homotopy fibre $F$.
The inclusion $i_1\colon \red{{X}}^{\sigma_1} \to \red{\und{X}}^K$ satisfies $p_1 \circ i_1 = \mathrm{id}$, so
$\red{{X}}^{\sigma_1}$ is a retract of $\red{\und{X}}^K$. Hence
we may consider
$i_1$ to be a homotopy section of the homotopy fibration $p_1$. By \thmref{thm:glo},
\begin{equation}\label{eqGLO}
\cat(\red{{X}}^{\sigma_1} \times F \times (\red{\und{X}}^K)^{r-2}) \leq \TC_r(\red{\und{X}}^K).
\end{equation}
The inclusion $i_2\colon \red{{X}}^{\sigma_2} \to \red{\und{X}}^K$ (with associated retraction $p_2\colon \red{\und{X}}^K \to \red{{X}}^{\sigma_2}$) \red{has} $p_1 \circ i_2 = *$, since $\sigma_1 \cap \sigma_2 = \varnothing$, so $i_2$ factors \red{up to homotopy} through the fibre $F$. Moreover, composing the resulting \red{homotopy class ${X}^{\sigma_2}\to F$} with the fibre
inclusion and $p_2$ gives the identity on $\red{{X}}^{\sigma_2}$ \red{(for $p_2\circ i_2=\text{id}$)}, so $\red{{X}}^{\sigma_2}$ is a \red{homotopy} retract of $F$. Of course, it is also the case that
$\red{{X}}^{\sigma_3} \times \cdots \times \red{{X}}^{\sigma_r}$ is a retract of $(\red{\und{X}}^K)^{r-2}$ as well, so $\red{{X}}^{\sigma_1} \times \cdots \times \red{{X}}^{\sigma_r}$
is a \red{homotopy} retract of $\red{{X}}^{\sigma_1} \times F \times (\red{\und{X}}^K)^{r-2}$. \red{Thus}
$$\cat(\red{{X}}^{\sigma_1} \times \cdots \times \red{{X}}^{\sigma_r}) \leq \cat(\red{{X}}^{\sigma_1} \times F \times (\red{\und{X}}^K)^{r-2}) \leq \TC_r(\red{\und{X}}^K),$$
\red{which yields the result.}
\end{proof}


\section{Real projective spaces}\label{sec:rpss}
For a family $\und{P}=\{P^{n_i}\}_{i=1}^m$ of real projective spaces, and a simplicial complex $K$ with vertices $[m]$, set
\begin{equation}\label{lanorma}
\norma{(n_1,\ldots,n_m)}:=\max\left\{\sum_{i\in \sigma_{1}\bigtriangleup\sigma_{2}}\!\!\!n_i+\sum_{i\in \sigma_1\cap\sigma_{2}}\!\!\TC(\RP^{n_i}):\sigma_{1},\sigma_{2}\in K\right\},
\end{equation}
where $\sigma_{1}\bigtriangleup\sigma_{2}=(\sigma_{1}\setminus\sigma_{2})\cup(\sigma_{2}\setminus\sigma_{1})$, the symmetric difference of $\sigma_1$ and $\sigma_2$. In this terms, Theorem~\ref{Upper bound TC of Z(Pn,K)} reads:
\begin{equation}\label{traducido}
\TC(\und{P}^K)\leq\norma{(n_1,\ldots,n_m)}.
\end{equation}

Since $\TC(P^n)<2n=2\,\cat(P^n)$, Theorem~\ref{Upper bound TC of Z(Pn,K)} is an improvement (for families of real projective spaces) of the upper estimate in Corollary~\ref{cor:TCupbound} (for $r=2$). On the other hand, by taking $\sigma_1=\sigma_2$ in~(\ref{lanorma}), we see that the value that Theorem~\ref{thm:polyTC} would predict for $\TC(\und{P}^K)$ could turn out to be smaller than the actual value of this invariant (of course hypothesis (1) in Theorem~\ref{thm:polyTC} does not hold for the family $\und{P}$). Example~\ref{ejemplomezclado} below illustrates a concrete such situation.

The inequality in~(\ref{traducido}) is sharp in a number of situations. The simplest one is when each $P^{n_i}$ is parallelizable, for then $\TC(P^{n_i})=n_i$, so equality in Theorem~\ref{Upper bound TC of Z(Pn,K)} follows from Corollary~\ref{cor:glo} and the obvious fact that families of real projective spaces are LS-logarithmic. More generally, Proposition~\ref{mixedlow} implies that~(\ref{traducido}) is optimal as long as, for each $i$, $\TC(P^{n_i})$ agrees with $\zcl(P^{n_i})$, the zero-divisors cup-length of $P^{n_i}$ with mod-2 coefficients, i.e., the maximal number of elements $z_j$ in the kernel of the cup-multiplication $H^*(P^{n_i})\otimes H^*(P^{n_i})\to H^*(P^{n_i})$ with $\prod z_j\neq0$.

\begin{remark}
It is known that the equality $\TC(P^n)=\zcl(P^n)$ holds precisely in the following (non-mutually exclusive) cases: (1) $n\in\{1,3,7\}$, i.e., $P^n$ is parallelizable; (2) $n=2^e+\delta$ with $e\geq1$ and $\delta=0,1$; (3) $n=6$.
\end{remark}

Since $n\leq\zcl(P^n)\leq\TC(P^n)$ for any $n$, both maximums in Theorem~\ref{Upper bound TC of Z(Pn,K)} and Proposition~\ref{mixedlow} are always realized by maximal simplices $\sigma_1$ and $\sigma_2$ of $K$, i.e., simplices that are not contained in any other simplex of $K$. Nontheless, it can be the case that these maximums are realized by sums that necessarily involve \emph{both} $\cat$ and $\TC$ terms. In other words, our estimates exhibit a mixed $\cat$/$\TC$ behavior not present for RAA groups. For instance:

\begin{example}\label{ejemplomezclado}
Let $K=\partial\Delta^2$, the simplicial complex with vertices $\{1,2,3\}$ and facets $\{1,2\}$, $\{1,3\}$ and $\{2,3\}$, and take $n_1=n_2=n_3=2^e$ with $e\geq0$. Since $\TC(P^{2^e})=2^{e+1}-1$, Theorem~\ref{Upper bound TC of Z(Pn,K)} and Proposition~\ref{mixedlow} yield $\TC(\und{P}^K)=2^{e+2}-1$, a maximum that is realized only with \emph{different} maximal simplices $\sigma_1$ and $\sigma_2$, so that $|\sigma_1\cap\sigma_2|=1$ and $|\sigma_{1}\bigtriangleup\sigma_{2}|=2$.
\end{example}

Proposition~\ref{dranishnikovgeneralizado} below shows that~(\ref{traducido}) is also sharp if $\dim(K)=0$ (for unrestricted $n_i$'s, and with no mixed $\cat$/$\TC$ behavior). This is proved using~\cite[Theorem~6]{DraSad} by induction on the number of vertices of $K$. The easy proof details are left as an exercise for the interested reader.

\begin{proposition}\label{dranishnikovgeneralizado}
If $n_1\geq n_2\geq \cdots\geq n_m$, then
$$\TC\left(\bigvee_iP^{n_i}\right)=\max\left\{n_1+n_2,\TC(P^{n_1})\right\}.$$
\end{proposition}

\begin{proof}[Proof of Theorem~\ref{higherproj}]
This is a consequence of Theorem~\ref{thm:polyTC}; we need only check the three relevant hypotheses. The LS-logarithmicity hypothesis is obvious, while the $\TC_r=r\,\cat$ condition is given by~\cite[Equation (5.2) and Theorem~5.7]{CGGGL}. Namely, under the present hypotheses (large $r$),
$$
r\cdot n_i=\TC_r(P^{n_i})=r\,\cat(P^{n_i})=\zcl_r(P^{n_i}),
$$
for all $i$, where $\zcl_r$ stands for $r$-th zero-divisors cup-length, see~\cite[Definition~3.8]{BGRT}. Lastly, the $\TC_r$-hypothesis follows from
$$
r\!\left(\sum_{j=1}^\ell n_{i_j}\!\!\right)\!\hspace{-.2mm}\geq \hspace{-.2mm}\TC_r\!\left(\prod_{j=1}^\ell\! P^{n_{i_j}}\!\!\right)\!\hspace{-.2mm}\geq\hspace{-.2mm}\zcl_r\!\left(\prod_{j=1}^\ell\! P^{n_{i_j}}\!\!\right)\!\hspace{-.2mm}\geq\hspace{-.2mm}\sum_{j=1}^\ell\zcl_r(P^{n_{i_j}})\hspace{-.2mm}=\hspace{-.2mm}r\!\left(\sum_{j=1}^\ell n_{i_j}\!\!\right)\hspace{-.8mm},
$$
where the first inequality comes from Theorem~\ref{thm:bounds}, the second one comes from~\cite[Theorem~3.9]{BGRT}, and the third one comes from~\cite[Lemma 2.1]{FaCo}.
\end{proof}

\subsection{Proof of Proposition~\ref{mixedlow}}
Recall from \cite[Theorem 2.35]{Gitler} that the mod-2 cohomology ring of $\und{P}^K$ is
$$
H^{*}(\und{P}^K)\cong \bigotimes_{i=1}^{m}H^{*}(\RP^{n_i})/I_K.
$$
Here $I_K$ is the generalized Stanley-Reisner ideal generated by all elements $x_{r_1}\otimes x_{r_2}\otimes\cdots\otimes x_{r_t},$ satisfying $x_{r_{i}}\in \overline{H}^{\,*}(\RP^{n_{r_i}})$ and $\{r_{1},\ldots,r_{t}\}\notin{K}$. For each $i\in[m]$, let $u_{i}\in H^*(\und{P}^K)$ be the pullback of the generator of $H^{1}(\RP^{n_i})$ under the canonical projection $\und{P}^K\to \RP^{n_i}$ onto the $i$-th coordinate. In these terms, $I_K$ is generated by the powers $u_{i}^{n_{i}+1}$ and the products $u_{r_1}\cdots u_{r_t}$ with $\{r_1,\ldots,r_t\}\notin{K}$.
In particular, a graded basis for $H^{*}(\und{P}^K)$ is given by the monomials
$
u_{1}^{e_1}\cdots u_{m}^{e_m}
$
having $0\leq e_{i}\leq n_i$ and $\{i: e_{i}>0\}\in{K}$. Note that $u_{i}\otimes 1+1\otimes u_{i}\in [H^{*}(\und{P}^K)]^{\otimes 2}$ is a zero divisor for each $i\in\{1,\ldots,m\}$.

Let $\sigma_{1},\sigma_{2}\in{K}$, say with $\sigma_{1}\setminus\sigma_{2}=\{i_1,\ldots,i_{r}\}$, $\sigma_{2}\setminus\sigma_{1}=\{j_1,\ldots,j_{s}\}$, and $\sigma_{1}\cap\sigma_{2}=\{k_{1},\ldots,k_{w}\}$. It suffices to show that, in $[H^{*}(\und{P}^K)]^{\otimes 2}$,
\begin{equation}\label{new lower bound}
\left(\prod_{i\in\sigma_{1}\bigtriangleup\sigma_2}\!\!(u_{i}\otimes 1+1\otimes u_{i})^{n_i}\!\right)\!\left(\,\prod_{i\in\sigma_{1}\cap\sigma_2}\!\!(u_{i}\otimes 1+1\otimes u_{i})^{\zcl(P^{n_i})}\!\right)\neq 0.
\end{equation}
Expanding the first factor on the left-hand side of~(\ref{new lower bound}), we get
\begin{align}
\left(\sum_{\ell_{1}=0}^{n_{i_1}}\binom{n_{i_1}}{\ell_1}\right.&\left.u_{i_1}^{n_{i_1}-\ell_1}\otimes u_{i_1}^{\ell_1}\rule{0mm}{7.5mm}\right)
\cdots\left(\sum_{\ell_{r}=0}^{n_{i_r}}\binom{n_{i_r}}{\ell_r}u_{i_{r}}^{n_{i_r}-\ell_{r}}\otimes u_{i_{r}}^{\ell_{r}}\right)\nonumber\\
\cdot\left(\sum_{q_{1}=0}^{n_{j_1}}\right.&\left.\binom{n_{j_1}}{q_1}u_{j_1}^{n_{j_1}-q_1}\otimes u_{j_1}^{q_1}\rule{0mm}{7.5mm}\right)
\cdots\left(\sum_{q_{s}=0}^{n_{j_s}}\binom{n_{j_s}}{q_s}u_{j_s}^{n_{j_s}-q_{s}}\otimes u_{j_{s}}^{q_{s}}\right)\nonumber\\
=&\left(\sum_{\substack{0\leq \ell_{t}\leq n_{i_t}\\ 1\leq t\leq r}}\binom{n_{i_1}}{\ell_1}\cdots\binom{n_{i_r}}{\ell_r}u_{i_1}^{n_{i_1}-\ell_1}\cdots u_{i_r}^{n_{i_r}-\ell_r}\otimes u_{i_1}^{\ell_1}\cdots u_{i_r}^{\ell_r}\right)\label{primerfactor}\\
&\cdot\left(\sum_{\substack{0\leq q_{t}\leq n_{j_t}\\ 1\leq t\leq s}}\binom{n_{j_1}}{q_1}\cdots\binom{n_{j_s}}{q_s}u_{j_1}^{n_{j_1}-q_1}\cdots u_{j_s}^{n_{j_s}-q_s}\otimes u_{j_1}^{q_1}\cdots u_{j_s}^{q_s}\right).\label{segundofactor}
\end{align}
On the other hand, it is elementary to see that $\zcl(P^{n_i})=2^{\theta_i}-1$ where $2^{\theta_{i}-1}\leq n_i<2^{\theta_{i}}$. In particular $n_i\leq 2^{\theta_i}-1<2n_{i}$. With this in mind, the second factor on the right-hand side of~(\ref{new lower bound}) becomes
\begin{align}
\left(\sum_{h_{1}=2^{\theta_{k_1}}-1-n_{k_1}}^{n_{k_1}}\right.&\left.\!\!\!\!\!\!\!\!\!\!u_{k_1}^{2^{\theta_{k_1}}-1-h_1}\otimes u_{k_1}^{h_1}\rule{0mm}{9mm}\right)\cdots\left(\sum_{h_{w}=2^{\theta_{k_w}}-1-n_{k_w}}^{n_{k_w}}\!\!\!\!\!\!\!\!\!\!u_{k_w}^{2^{\theta_{k_w}}-1-h_w}\otimes u_{k_w}^{h_w}\right)\nonumber\\
=&\left(\sum_{\substack{2^{\theta_{k_t}}-1-n_{k_t}\leq h_{t}\leq n_{k_t}\\ 1\leq t\leq w}}\!\!\!\!\!\!\!\!\!\!u_{k_1}^{2^{\theta_{k_1}}-1-h_1}\cdots u_{k_w}^{2^{\theta_{k_w}}-1-h_w}\otimes u_{k_1}^{h_1}\cdots u_{k_w}^{h_w}\right).\label{tercerfactor}
\end{align}

The conclusion then follows by observing that, in the product of the expressions above, the basis element
$$
u_{i_1}^{n_{i_1}}\cdots u_{i_r}^{n_{i_r}}u_{k_1}^{n_{k_1}}\cdots u_{k_w}^{n_{k_w}}\otimes u_{j_1}^{n_{j_1}}\cdots u_{j_s}^{n_{j_s}}u_{k_1}^{2^{\theta_{k_1}}-1-n_{k_1}}\cdots u_{k_w}^{2^{\theta_{k_w}}-1-n_{k_w}}
$$
arises only from the product of:
\begin{itemize}
\item $u_{i_1}^{n_{i_1}}\cdots u_{i_r}^{n_{i_r}}\otimes 1\,$ in (\ref{primerfactor}), with $\ell_{t}=0$ for all $t\in\{1,\ldots,r\}$;
\item $1\otimes u_{j_1}^{n_{j_1}}\cdots u_{j_s}^{n_{j_s}}\,$ in (\ref{segundofactor}), with $q_{t}=n_{j_t}$ for all $t\in\{1,\ldots,s\}$;
\item $u_{k_1}^{n_{k_1}}\cdots u_{k_w}^{n_{k_w}}\otimes u_{k_1}^{2^{\theta_{k_1}}-1-n_{k_1}}\cdots u_{k_w}^{2^{\theta_{k_w}}-1-n_{k_w}}\,$ in (\ref{tercerfactor}), with $h_{t}=2^{\theta_{k_t}}-1-n_{k_t}$ for all $t\in\{1,\ldots,w\}$.
\end{itemize}

\subsection{Axial and nonsingular maps}
\label{axial and nonsingular maps}
The rest of the paper is devoted to the proof of Theorem~\ref{Upper bound TC of Z(Pn,K)}. For this, we need the full force of the concept of axial map and its relationship to the topological complexity of real projective spaces. We review the needed details.

\begin{definition}\label{nonsingular map}
Let $n$ and $k$ be nonnegative integers. A continuous map
$
f\colon \real{n+1}\times\real{n+1}\to\real{k+1}
$
is called {nonsingular} provided
(a) $f(\lambda x,\mu y)=\lambda\mu f(x,y)$ for all $x,y\in\real{n+1}$, $\lambda,\mu\in\real{}$, and (b)
$f(x,y)=0$ implies either $x=0$ or $y=0$.
If in addition $f(x,\ast)=x$ and $f(\ast,x)=x$ for all $x\in\mathbb{R}^{n+1}$, $f$ is called a {strong nonsingular map}. Here, $\ast$ stands for the base point $(1,0,\ldots,0)$ of $\real{n+1}$. Furthermore, we use the canonical embedding $\real{n+1}\hookrightarrow\real{k+1}$ to identify $(x_1,\ldots,x_{n+1})\in\real{n+1}$ with $(x_1,\ldots,x_{n+1},0,\ldots,0)\in\real{k+1}$.
\end{definition}

It is classical that a nonsingular map as above exists only for $k\geq n$, and that $k=n$ is possible only for $n\in\{1,3,7\}$, in which case the complex, quaternion, and octonion multiplications can be used. In the latter cases, the multiplication $\omega\cdot z$ yields a strong nonsingular map, while the multiplication $\omega\cdot\overline{z}$ yields a nonsingular map $f\colon \real{n+1}\times\real{n+1}\to\real{n+1}$ satisfying $$f(x,x)=(\lambda_x,0,\ldots,0)$$ with $\lambda_x>0$ for $x\neq 0$. From the polyhedral product viewpoint, the latter condition allows us to detect the $\cat$/$\TC$ mixed behavior illustrated in Example~\ref{ejemplomezclado} ---see the discussion following the next definition.

\begin{definition}\label{axial map}
Let $n$ and $k$ be nonnegative integers. A continuous map
$
\alpha:\RP^{n}\times\RP^{n}\to \RP^{k}
$
is called {axial} if its restrictions $\alpha|_{\RP^{n}\times\ast}$ and $\alpha|_{\ast\times\RP^{n}}$ are detected in mod 2 cohomology (so $k\geq n$ is forced). If in fact $\alpha|_{\RP^{n}\times\ast}$ and $\alpha|_{\ast\times\RP^{n}}$ are the equatorial inclusions, i.e., if $\alpha(A,\ast)=A$ and $\alpha(\ast,A)=A$ hold for any $A\in\RP^{n}\subset\RP^{k}$, then $\alpha$ is called a {strong axial map}. Here, $\ast$ is the base point of $\RP^{\ell}$ given by the equivalence class of $(1,0,\ldots,0)\in\real{\ell+1}$.
\end{definition}

It is elementary to see that a nonsingular map $f\colon\mathbb{R}^{n+1}\times\mathbb{R}^{n+1}\to \mathbb{R}^{k+1}$ induces and is induced by an axial map $\alpha\colon \RP^{n}\times\RP^{n}\longrightarrow \RP^{k}$, and that, in these conditions, $f$ is strong if and only if $\alpha$ is strong. In turn, a strong axial map $\RP^{n}\times\RP^{n}\to\RP^{k}$ can be constructed out of a smooth immersion $\RP^{n}\looparrowright\mathbb{R}^{k}$ (\cite[Theorem 2.1]{Sanderson}; note that $k>n$ is forced). Conversely, for $k>n$, the existence of a (not necessarily strong) axial map $\RP^{n}\times\RP^{n}\to \RP^{k}$ implies the existence of a smooth immersion $\RP^{n}\looparrowright \mathbb{R}^{k}$ (\cite{AGJ}). The ``strong'' requirement, which is then free when $k>n$, is central in what follows. (The situation $n=k$ is special in that there cannot be a smooth immersion $\RP^{n}\looparrowright \mathbb{R}^{n}$, and yet an axial map $\RP^{n}\times\RP^{n}\to \RP^{n}$ exists if and only if $n\in\{1,3,7\}$.)

These facts imply that, for $n\neq 1,3,7$, $\TC(\RP^{n})=\imm(\RP^{n})$, the Euclidean immersion dimension of $P^n$ (\cite[Theorem 7.1]{FTY}). A key point in the argument is the fact that, if $k>n$, an axial map $P^n\times P^n\to P^k$ is forced to be null-homotopic on the diagonal. The combination of (1) such a behavior with (2) the strong condition for axial maps yields the main ingredient in the proof of Theorem~\ref{Upper bound TC of Z(Pn,K)}. Actually, the strong property for axial maps is directly responsible for the mixed $\cat$/$\TC$ behavior ---and resulting apparent sharpness--- of Theorem~\ref{Upper bound TC of Z(Pn,K)} (see~(\ref{comportamientomezclado}) below).

The following fact is adapted from \cite[Lemma 5.7]{FTY}. We leave the proof as a standard algebraic topology exercise for the reader.

\begin{lemma}\label{strong axial map and diagonal}
For $n\neq 1,3,7$, there is an axial map $\alpha:\RP^{n}\times\RP^{n}\to\RP^{\TC(\RP^{n})}$ that is strong and satisfies $\alpha(A,A)=\ast$ for all $A\in P^n$. Consequently,
there is a strong nonsingular map $f:\real{n+1}\times\real{n+1}\to \real{\TC(\RP^{n})+1}$ such that $f(x,x)=(\lambda_{x},0,\ldots,0)$, $\lambda_{x}\geq 0$, for all $x\in\real{n+1}$, and $\lambda_{x}= 0$ only for $x=0$.
\end{lemma}

\subsection{Proof of Theorem~\ref{Upper bound TC of Z(Pn,K)}}
\label{general upper bound}
For each $i\in\{1,\ldots,m\}$, fix once and for all a nonsingular map
$f_{i}:\mathbb{R}^{n_{i}+1}\times\mathbb{R}^{n_{i}+1}\to \mathbb{R}^{\TC(\RP^{n_i})+1}$ such that
\begin{equation}\label{diagonal}
f_i(x,x)=(\lambda_{x},0,\ldots,0),\quad \lambda_{x}\geq 0,
\end{equation}
for all $x\in\real{n_i+1}$, with $\lambda_x>0$ for $x\neq 0$. Further, if $n_i\neq 1,3,7$, we also require $f_i$ to be strong:
\begin{equation}\label{strong}
f_i(\ast,x)=x=f_i(x,\ast), \mbox{ for all $x\in\real{n_{i}+1}$.}
\end{equation}

For $k\in\{0,1,\ldots,\TC(\RP^{n_i})\}$, let $f_{ik}$ denote the $(k+1)$-st coordinate map of $f_{i}$, and consider the subsets $V_{ik}\subset P^{n_i}\times P^{n_i}$ defined by
\begin{align*}
V_{i0}&=\{(A,B)\colon f_{i0}(a,b)\neq 0\texto{for some}(a,b)\in A\times B\}\\
V_{ik}&=\{(A,B)\colon A\neq B\texto{and}f_{ik}(a,b)\neq 0\texto{for some}(a,b)\in A\times B\}, \mbox{ for $k>0$.}
\end{align*}
Observe that $\{V_{i0},V_{i1},\ldots,V_{i\TC(\RP^{n_i})}\}$ is an open cover of $\RP^{n_i}\times\RP^{n_i}$, with the diagonal of $\RP^{n_i}$ contained in $V_{i0}$, in view of~(\ref{diagonal}).

\begin{remark}\label{alafty}
As observed in~\cite{FTY}, on each set $V_{ik}$, the end-points evaluation map $(\RP^{n_i})^{[0,1]}\to \RP^{n_i}\times\RP^{n_i}$ has a continuous local section $\lambda_{ik}:V_{ik}\to(\RP^{n_i})^{[0,1]}$ defined as follows: In the diagonal, $\lambda_{0k}(A,A)$ is the constant path at $A$. Otherwise, for $A\neq B$, choose unit vectors $a\in A$ and $b\in B$ with $f_{ik}(a,b)>0$. The only other pair of unit vectors satisfying the latter condition is $(-a,-b)$. Both pairs $(a,b)$ and $(-a,-b)$ determine the same orientation of the plane spanned by $A$ and $B$. We then set $\lambda_{ik}(A,B)$ to be the rotation with constant velocity from $A$ toward $B$ in the plane spanned by $A$ and $B$, following the direction determined by the above orientation. Notice that the continuity of $\lambda_{i0}$ on $V_{i0}$ follows from~(\ref{diagonal}), as $\lambda_{i0}(A,A)$ is constant.
\end{remark}

We replace the covering $\{V_{ik}\}_k$ of each $P^{n_i}$ by one which consists of pairwise disjoint sets. Put $U_{ik}:=V_{ik}\setminus (V_{i0}\cup\cdots\cup V_{i(k-1)})$
(so $U_{i0}=V_{i0}$). We say that an ordered pair of lines $(A,B)$ in $\real{n_{i}+1}$ produces $k$ zeroes ($k\in\{0,1,\ldots,\TC(\RP^{n_i})\}$) when $(A,B)\in U_{ik}$. The justification of the latter convention comes by observing that $(A,B)\in U_{ik}$ if only if $$f_{i0}(a,b)=\cdots=f_{i(k-1)}(a,b)=0\neq f_{ik}(a,b)$$ for some (and therefore any) vectors $a\in A$ and $b\in B$. The number of zeroes produced by $(A,B)$ is denoted $Z(A,B)$.

For simplicity we write $X=\und{P}^K$, the polyhedral product in Theorem~\ref{Upper bound TC of Z(Pn,K)}, but keep the notation ${P}^\sigma$ for the ``subproducts'' in Corollary~\ref{cor:glo}. An element $(A_1,A_2)$ of $X^{2}$ will be thought of as a matrix of size $m\times 2$, i.e.,
\begin{equation}
\nonumber
(A_1,A_2)=
\begin{pmatrix}
A_{11}&A_{12}\\
\vdots&\vdots\\
A_{m1}&A_{m2}
\end{pmatrix},
\end{equation}
where each column belongs to $X$, say $(A_{1j},\ldots,A_{mj})\in {P}^{\sigma_{j}}$ with $\sigma_{1},\sigma_{2}\in{K}$. Each row $(A_{i1},A_{i2})$ of the matrix $(A_1,A_2)$ lies in a unique set $U_{ik}$ with $k=Z(A_{i1},A_{i2})$. The number of zeroes determined by $(A_1,A_2)$, denoted by $Z(A_1,A_2)$, is the sum of zeroes produced by the rows of $(A_1,A_2)$,
$$
Z(A_1,A_2)=\sum_{i=1}^{m}Z(A_{i1},A_{i2}).
$$

It is apparent that $Z(A_1,A_2)\leq\sum_{i\in\sigma_{1}\cup\sigma_{2}}\TC(\RP^{n_i})$. However, in view of~(\ref{strong}), the number of zeroes produced by rows of type either $(A_{i1},\ast)$ or $(\ast, A_{i2})$ of $(A_1,A_2)$ calls for a more careful consideration. Namely, note that $Z(A_{i1},A_{i2})=Z(A_{i1},\ast)\leq n_i$ for $i\in \sigma_{1}\setminus\sigma_{2}$, because $f_{i}$ fulfills~(\ref{strong}) when $n_i\neq 1,3,7$ (recall that $n_i=\TC(\RP^{n_i})$ if $n_i=1,3,7$). Likewise, if $i\in \sigma_{2}\setminus\sigma_{1}$, then $Z(A_{i1},A_{i2})=Z(\ast,A_{i2})\leq n_i$. Therefore
\begin{equation}\label{comportamientomezclado}
Z(A_1,A_2)\leq\sum_{i\in \sigma_{1}\bigtriangleup\sigma_{2}}n_i+\sum_{i\in \sigma_1\cap\sigma_{2}}\TC(\RP^{n_i}),
\end{equation}
and we have proved:
\begin{proposition}
The sets $W_{j}=\{(A_1,A_2)\in X^{2}: Z(A_1,A_2)=j\}$, with $j\in\{0,1,\ldots,\norma{(n_1,\ldots,n_m)}\}$, form a pairwise disjoint covering of $X^{2}$.
\end{proposition}

\begin{remark}
The sets $W_j$ fail to give an open covering. However, they can be used to calculate the \emph{generalized} topological complexity of $X$, $\TC_{g}(X)$ (not to be confused 
with a higher $\TC_g$). 
This generalized topological complexity is defined in the same way as $\TC(X)$, except that the openness condition imposed on the cover of $X^{2}$ is not required. In our case, $\TC(X)$ and $\TC_{g}(X)$ agree since $X$ is a $\CW$ complex (see~\cite[Corollary 2.8]{JoseManuel}).
\end{remark}

The proof of Theorem~\ref{Upper bound TC of Z(Pn,K)} will be complete once a local rule (i.e., a local section for the end-points evaluation map $X^{[0,1]}\to X\times X$) is constructed on each $W_{j}$. This will be achieved by splitting $W_{j}$ into topologically disjoint subsets (Proposition~\ref{topological union proof} below), and then showing that each such subset admits a local rule.

A partition of an integer $j\in\{0,1,\ldots, \norma{(n_1,\ldots,n_m)}\}$ is an ordered tuple $(j_1,\ldots, j_{m})$ of integers satisfying $j=j_{1}+\cdots+j_{m}$ and $0\leq j_{i}\leq \TC(\RP^{n_i})$ for each $i=1,\ldots,m$. For such a partition we set
\begin{equation}
\nonumber
W_{(j_1,\ldots, j_{m})}=\{(A_1,A_2)\in X^{2}: Z(A_{i1},A_{i2})=j_i \texto{ for each }i\in[m]\}.
\end{equation}
We show that the resulting set-theoretic partition $W_{j}=\displaystyle\bigsqcup W_{(j_1,\ldots, j_{m})}$, where the union runs over the partitions $(j_1,\ldots,j_m)$ of $j$, is topological, i.e., $W_{j}$ has the weak topology determined by the several $W_{(j_1,\ldots, j_{m})}$. In fact:

\begin{proposition}\label{topological union proof}
Two different partitions $(j_1,\ldots, j_{m})$ and $(r_1,\ldots, r_{m})$ of the same $j\in\{0,1,\ldots,\norma{(n_1,\ldots,n_m)}\}$ have
$$
W_{(j_1,\ldots, j_{m})}\cap\overline{W_{(r_1,\ldots, r_{m})}}=\varnothing=\overline{W_{(j_1,\ldots, j_{m})}}\cap W_{(r_1,\ldots, r_{m})}.
$$
\end{proposition}

\begin{proof}
It suffices to show the first equality assuming $j_{\ell}<r_{\ell}$ for some $\ell\in[m]$. For elements $(A_{1},A_{2})\in W_{(j_1,\ldots,j_{m})}$ and $(B_{1},B_{2})\in W_{(r_1,\ldots,r_{m})}$ we have
\begin{align*}
&(A_{\ell 1},A_{\ell 2})\in U_{\ell j_{\ell}}\Rightarrow
f_{\ell j_{\ell}}(a_{\ell 1},a_{\ell 2})\neq 0 \text{ \ for any \ }(a_{\ell 1},a_{\ell 2})\in A_{\ell 1}\times A_{\ell 2}\\
&(B_{\ell 1},B_{\ell 2})\in U_{\ell r_{\ell}}\Rightarrow
f_{\ell j_{\ell}}(b_{\ell 1},b_{\ell 2})=0 \text{ \ for any \ }
(b_{\ell 1}, b_{\ell 2})\in B_{\ell 1}\times B_{\ell 2}.
\end{align*}
The assertion then follows since the latter condition is inherited by elements of $\overline{W_{(r_1,\ldots,r_{m})}}$.
\end{proof}

In the rest of the paper we construct a local rule on each $W_{(j_1,\ldots,j_m)}$. For $i\in\{1,\ldots,m\}$, consider the canonical Riemannian structure on the unit $n_i$-sphere $S^{n_i}$. Since the antipodal involution ${S}^{n_i}\to{S}^{n_i}$ is a fixed-point free properly discontinuous isometry, $\RP^{n_i}$ inherits a canonical quotient Riemannian structure $g_i$. Let $L_{i}(\gamma)$ denote the resulting length of a smooth curve $\gamma$ in $\RP^{n_i}$, and let
$$
d_{i}(x,y)=\inf \{L_{i}(\gamma):\gamma\texto{is a geodesic on}\RP^{n_i}\texto{from}x\texto{to}y\}
$$
be the associated metric. Note that the curves $\lambda_{ik}(A,B)$ described in Remark~\ref{alafty} are geodesic on $\RP^{n_i}$. Without loss of generality we can assume that each $g_i$ is normalized so that any geodesic $\gamma$ on $\RP^{n_i}$ satisfies $L_{i}(\gamma)\leq 1/2$.

As a first step, we reparametrize the families of paths $\lambda_{ik}$. Explicitly, for $k\in\{0,1,\ldots,\TC(\RP^{n_i})\}$, consider the section $\tau_{ik}:U_{ik}\to (\RP^{n_i})^{[0,1]}$ of the end-points evaluation map $(\RP^{n_i})^{[0,1]}\to \RP^{n_i}\times\RP^{n_i}$ given by
\[
\tau_{ik}(A_{1},A_{2})(t)=\begin{cases}
A_1,&\texto{if } d_{i1}+d_{i2}=0;\\
\lambda_{ik}(A_1,A_2)(\frac{t}{d_{i1}+d_{i2}}),&\texto{if } 0\leq t\leq (d_{i1}+d_{i2})\neq 0;\\
A_{2},&\texto{if } 0\neq (d_{i1}+d_{i2})\leq t\leq 1,
\end{cases}
\]
for $(A_1,A_2)\in U_{ik}$,  where $d_{ij}=d_{i}(A_j,\ast)$, $j=1,2$. Note that $\tau_{ik}$ is continuous on the open subset of $U_{ik}$ determined by the condition $d_{i1}+d_{i2}\neq 0$. The latter open subset of $U_{ik}$ equals in fact $U_{ik}$ unless $k=0$, so that $\tau_{ik}$ is continuous on the whole $U_{ik}$ for $k\in\{1,\ldots,\TC(\RP^{n_i})\}$. The continuity of $\tau_{i0}$ on $U_{i0}$ follows from the continuity of $\lambda_{i0}$ and the fact that $\lambda_{i0}(A,A)$ is the constant path for all $A\in\RP^{n_i}$. Note in addition that $\tau_{ik}$ reaches its end point at time $d_{i1}+d_{i2}$ (this is function-wise speaking). In particular $\tau_{ik}(A_1,\ast)$ reaches $\ast$ at time $d_{i1}$.

The map $\varphi:X^{2}\to(\prod_{i=1}^{m}\RP^{n_i})^{[0,1]}$ we are after is defined by
$$
\varphi(A_{1},A_{2})=(\varphi_{1}(A_{11},A_{12}),\ldots,\varphi_{m}(A_{m1},A_{m2})),
$$
with $i$-th coordinate $\varphi_{i}(A_{i1},A_{i2})$, the path in $\RP^{n_i}$ from $A_{i1}$ to $A_{i2}$ given by
\begin{equation}\label{final movement}
\varphi_{i}(A_{i1},A_{i2})(t)=\begin{cases}
A_{i1},&\texto{if } 0\leq t\leq t_{A_{i1}};\\
\mu(A_{i1},A_{i2})(t-t_{A_{i1}}),&\texto{if } t_{A_{i1}}\leq t\leq 1.
\end{cases}
\end{equation}
Here, $t_{A_{i1}}=1/2-d_{i}(A_{i1},\ast)$ and
\begin{equation}\label{final formulas}
\mu(A_{i1},A_{i2})=\begin{cases}
\tau_{i0}(A_{i1},A_{i2}),&\texto{if } (A_{i1},A_{i2})\in U_{i0};\\
\hspace{1mm}\vdots&\hspace{2mm}\vdots\\
\tau_{i\TC(\RP^{n_i})}(A_{i1},A_{i2}),&\texto{if } (A_{i1},A_{i2})\in U_{i\TC(\RP^{n_i})}.
\end{cases}
\end{equation}

\smallskip
By construction, $\varphi$ is a global section of the end-points evaluation map
$$
\left(\prod_{i=1}^{m}\RP^{n_i}\right)^{[0,1]}\!\!\longrightarrow \,\left(\prod_{i=1}^{m}\RP^{n_i}\right)^{2}\!\!.
$$
Although $\varphi$ is not globally continuous, the restriction of $\varphi$ to each $W_{(j_1,\ldots, j_{m})}$, where $(j_1,\ldots,j_{m})$ is a partition of $j\in\{0,1,\ldots, \norma{(n_1,\ldots,n_m)}\}$, is continuous because formulas~(\ref{final formulas}) can be rewritten as
$$
\mu=\begin{cases}
\tau_{i0},&\texto{if } j_{i}=0;\\
\hspace{1mm}\vdots&\hspace{2mm}\vdots\\
\tau_{i\TC(\RP^{n_i})},&\texto{if } j_{i}=\TC(\RP^{n_i}).
\end{cases}
$$
So, the crux of the matter is to show (in Proposition~\ref{porfin} below) that the restriction of $\varphi$ to each $W_{(j_1,\ldots,j_m)}$ lands in $X^{[0,1]}$.

\begin{remark}\label{explanation of formulas}
We spell out the motion planning instructions at the level of each factor in the polyhedral product. For $(A_{i1},A_{i2})\in U_{ik}$ with $k\in\{0,1,\ldots,\TC(\RP^{n_i})\}$, the navigational instruction is:
\begin{itemize}
\item stay at $A_{i1}$, for times $t\in[0,1/2-d_{i}(A_{i1},\ast)]$;
\item move from $A_{i1}$ to $A_{i2}$ at constant speed (recall the normalization) via $\tau_{ik}$, for times $t\in [1/2-d_{i}(A_{i1},\ast),1/2+d_{i}(A_{i2},\ast)]$;
\item stay at $A_{i2}$, for times $t\in[1/2+d_{i}(A_{i2},\ast),1]$.
\end{itemize}
\end{remark}
\begin{example}\label{from basepoint to A}
Suppose $A_{i1}=\ast$ and let $A_{i2}$ be any line through the origin in $\real{{n_i}+1}$. 
In this case, $t_{A_{i1}}=1/2-d_{i}(A_{i1},\ast)=1/2-0=1/2$. By Remark~\ref{explanation of formulas}, the path $\varphi_{i}(A_{i1},A_{i2})$ obeys the instructions:
\begin{itemize}
\item for $t\in[0,1/2]$, stay at $\ast$;
\item for $t\in[1/2,1/2+d_{i}(A_{i2},\ast)]$, move from $\ast$ to $A_{i2}$ at constant speed;
\item for $t\in[1/2+d_{i}(A_{i2},\ast),1]$, stay at $A_{i2}$.
\end{itemize}
\end{example}
\begin{example}\label{from A to basepoint}
Suppose $A_{i2}=\ast$ (so that $d_i(A_{i2},\ast)=0$) and let $A_{i1}$ be any line through the origin in $\real{n_{i}+1}$.
Then the path $\varphi_{i}(A_{i1},A_{i2})$ obeys the instructions:
\begin{itemize}
\item for $t\in[0,1/2-d_{i}(A_{i1},\ast)]$, stay at $A_{i1}$;
\item for $t\in[1/2-d_{i}(A_{i1},\ast),1/2]$, move from $A_{i1}$ to $\ast$ at constant speed;
\item for $t\in[1/2,1]$, stay at $\ast$.
\end{itemize}
\end{example}

\begin{proposition}\label{porfin}
The image of $\varphi$ is contained in $X^{[0,1]}$.
\end{proposition}
\begin{proof}
Let $(A_{1},A_{2})\in X^{2}$, say $(A_{1j},\ldots,A_{mj})\in {P}^{\sigma_{j}}$ with $\sigma_1,\sigma_{2}\in{K}$. Then, for $i\notin\sigma_{1}$, Example~\ref{from basepoint to A} shows that $A_{i1}=\ast$ keeps its position through time $t\leq 1/2$. Consequently $\varphi(A_1,A_{2})([0,1/2])\subseteq {P}^{\sigma_1}\subseteq X$. Likewise, for $i\notin\sigma_{2}$, Example~\ref{from A to basepoint} shows that the path $\varphi_{i}(A_{i1},A_{i2})$ has reached its final position $A_{i2}=\ast$ at time $1/2$, so that $\varphi(A_1,A_{2})([1/2,1])\subseteq {P}^{\sigma_{2}}\subseteq X$.
\end{proof}

\end{document}